\documentclass[12pt]{amsart}
\usepackage{amssymb}
\usepackage{enumerate}
\usepackage{array}
\newtheorem{thm}{Theorem}
\newtheorem{lem}[thm]{Lemma}

\newtheorem{cor}[thm]{Corollary}
\newtheorem{prop}[thm]{Proposition}

\theoremstyle{definition}  
\newtheorem{example}[thm]{Example}
\newtheorem{definition}[thm]{Definition}
\newtheorem{remark}[thm]{Remark}

\newcommand{\Z}{\mathbb Z}
\newcommand{\N}{\mathbb N}
\newcommand{\R}{\mathbb R}

\DeclareMathOperator{\Fac}{Fac}
\DeclareMathOperator{\End}{End}
\DeclareMathOperator{\Aut}{Aut}
\DeclareMathOperator{\Id}{Id}

\newcommand{\abs}[1]{\left| #1 \right|}

\usepackage{color}
\usepackage{hhline}

\usepackage{amssymb}
\usepackage{enumerate}
\usepackage{array}

\newcolumntype{C}{ >{\centering\arraybackslash} m{0.3cm}}

\begin{document}

\author{Clemens M\"{u}llner and  Reem Yassawi}
\address{Institut Camille Jordan, Universit\'{e} Claude Bernard Lyon 1, France, TU Wien, Austria \\
and \\
Institut Camille Jordan, Universit\'{e} Claude Bernard Lyon 1, France
}
\email{ clemens.muellner@tuwien.ac.at    \\ryassawi@gmail.com}

\keywords{substitution dynamical systems, automatic sequences, automorphisms}
\subjclass{Primary 37B10, 11B85,  Secondary 37B05, 37B15}
\thanks{ The authors were supported by     the European Research Council (ERC) under the European Union's Horizon 2020 research and innovation programme, Grant Agreement No 648132. 
The second author also thanks IRIF, Universit\'e Paris Diderot-Paris 7, for its hospitality and support.}

\title{Automorphisms of Automatic Shifts}
\maketitle

\begin{abstract}
In this article we continue the study of automorphism groups of constant length substitution shifts and also their topological factors. We show that up to conjugacy, all roots of the identity map are letter exchanging maps, 
and all other nontrivial automorphisms arise from {\em twisted} compressions of another constant length substitution. We characterise the group of  roots of the identity in both the measurable and topological setting. Finally, we show that any topological  factor of a constant length substitution shift is topologically conjugate to a constant length substitution shift via a letter-to-letter code.
\end{abstract}
\section{Introduction}

In this article we continue the study of the automorphism groups of constant length substitution shifts and complete some of our work in \cite{coven-quas-yassawi}. In particular, there we gave an algorithm for the computation of the automorphism group $\Aut(X_\theta,\sigma)$ of a constant length substitutional shift. Here, we 
ask what are the possible elements of this group, and also, what kinds of factors these shifts can have. In both cases, the answer is a form of  rigidity.

Our strategy in \cite{coven-quas-yassawi} was to attach a rational number $\kappa(\Phi)$ to each automorphism $\Phi$, which signifies  by how much points are being shifted while being acted on by $\Phi$ (Theorem \ref{thm:ethan}).  In particular,  a simple example of an automorphism whose $\kappa$-value  equals zero is that coming from a permutation on letters which plays well with the substitution $\theta$, by which we mean that it commutes with some power of 
the substitution. These automorphisms are natural and indicate that there is a symmetry in the substitution rule, up to a shuffling of letters. Now one can construct examples of substitutions with automorphisms whose $\kappa$-value is zero, but which are not of this nature.  However we show, in Theorem \ref{strongly-injective-rep}, that up to conjugacy, only the natural case manifests. 

What remains to consider are the automorphisms with a nonzero rational $\kappa$-value $\frac{1}{k}$. It is not difficult to create such situations. Namely, one starts with a substitution $\eta$, and one takes its {\em $k$-compression}. This is the process whereby we rewrite a sequence as a concatenation of words of length $k$, and where the shift map moves from one word to the next. Cobham \cite{Cobham} used this terminology, and he showed that constant length substitutions are robust under this operation: a $k$-compression of a length $r$ substitution shift  $(X_\eta,\sigma)$ gives us another length $r$ substitution shift $(X_\theta,\sigma)$. By construction,  $\Aut(X_\theta,\sigma)$ will contain a shadow of the shift on $X_\eta$, and this manifests as a   $k$-th root of the shift map on $X_\theta$. Can automorphisms $\Phi$ with $\kappa(\Phi)=\frac{1}{k}$ arise in any other way? We discover, in Theorem \ref{roots}, that up to conjugacy, the only other way that they can arise is if
we {\em twist} after we compress $\eta$ (Definition \ref{twist-def}). At some level, our results say that the automorphisms of our substitution shift $X_\theta$ are only constrained by fixed points, either of $\theta$, or of another substitution.

Can one hope that the letter exchanging maps of $\theta$ are related to those of $\eta$? In our world, this is generally not the case; we see this in Examples \ref{example-salvaged} and \ref{no-hope}. There is a fundamental obstruction discussed in Remark \ref{obstruction}. Hence we fall short of giving a general structure theorem. We could formulate one, but it would not be succinct.

We also characterise, in Theorem \ref{characterisation-of-kernel}, the letter exchanging maps that define automorphisms of $X_\theta$, generalising results due to Lemanczyk and Mentzen  in the bijective case \cite{L-M}. In particular, they characterised the measurable essential centralizer of a bijective substitution as being the centraliser of the group generated by the columns of $\theta$; the topological statement is identical. To characterise the letter exchanging maps of an arbitrary constant length substitution, we work with {\em minimal sets}, a notion that was implicit in the algorithm of \cite{coven-quas-yassawi}, but which was defined and used extensively by the first author in \cite[Section 2]{M-2017}. We obtain characterisations in both the topological setting, in Theorem \ref{characterisation-of-kernel}, and the measurable setting, in Theorem \ref{best-result-yet}.

These are our principal results, but on our trip to discovering them, we obtained some other results of interest along the way. The first, which follows from Lemma \ref{sliding-block-rep} and  
Theorem \ref{th:main}, tells us  that any topological shift factor of a constant length substitution shift $X_\theta$ is a letter-to-letter coding of a constant length substitution shift.  This answers a question of Durand and Leroy in \cite[Section 8]{Durand-Leroy}. It also suggests that the question that Coven, Dekking and Keane ask in \cite{CDK}, namely to list all {\em constant-length} substitutions whose shifts are conjugate to a given constant length substitution shift, is not constraining.

Finally, we investigate the relationship between $\Aut(X,\sigma)$ and the existence of a uniform-to-one factor for $X$. This has been studied by Herning \cite{Herning} in the case where $X = X_{\theta}$ and the factor is metrically a rotation. In Theorem \ref{automorphism-almost-automorphic-v2}, we show that $\Aut(X,\sigma)$ contains a square root of the identity if and only if $X$ has a two-to-one  factor. In Example
\ref{no-automorphism-example}, we show that there is no possibility of extending this result.

\section{Preliminaries}

Let $\mathcal A$ be a finite alphabet, with the discrete topology.
We endow   $\mathcal A^\mathbb Z$ 
with the product topology, and let 
  $\sigma:\mathcal A^\mathbb Z \rightarrow \mathcal A^\mathbb Z$
denote the (left) shift map.
We consider only infinite  {\em minimal} shifts $(X,\sigma)$, i.e. those such that the $\sigma$-orbit of any point  in $X$ is dense in $X$.  
The {\em language $\mathcal L_X$} of a shift $(X,\sigma)$ is the set of all finite words  $w=w_1\ldots w_n$ that appear as a subword of some $x$ in $X$, i.e. $w=x_j\ldots x_{j+n-1}$ for some $j$.

\subsubsection{Automorphism groups}

A {\em factor map}  from $(X, \sigma) $  to $(Y,\sigma)$ is a map $\Phi:X\rightarrow Y$ 
which is continuous and commutes with $\sigma$. If $X=Y$, $\Phi$ is called an endomorphism. If $\Phi$ is bijective, it is called a {\em conjugacy}, and 
if $\Phi$ is a self-conjugacy  then it is called an {\em automorphism}.
We say $\Phi:X\rightarrow X$ has (finite){\em order m} 
if $m$ is the least positive integer  such that
$\Phi^m$ is the identity map $\Id$, and we say that $\Phi$ is a 
{\em $k$-th root of the  shift if $\Phi^k=\sigma$}. Let ${\Aut}(X,\sigma)$ 
denote the automorphism group of $( X,\sigma)$; it contains the (normal) subgroup 
generated by the shift, denoted by $\langle\sigma\rangle$.

Let $X\subset \mathcal A^\Z$ and let $Y\subset \mathcal B^\Z$. By the Curtis--Hedlund--Lyndon theorem, for any factor map $\Phi:X\rightarrow Y$,
there exist $\ell$, $r$, and 
a map $f\colon \mathcal A^{\ell+r+1}\to \mathcal B$ with the property that
$(\Phi(x))_n= f(x_{n-\ell}, \ldots ,x_n, \ldots ,x_{n+r})$ for all $x$ and $n$.
Let $\ell,r$ be the smallest possible integers so that such an $f$ exists.
We call $\ell$ and $r$ 
the  left  and right radius of $\Phi$ respectively.   
We say $\Phi$ has  radius $R$ if
its left radius and right radius are both at most $R$.

\subsubsection{Constant length substitutions}
A  {\em substitution} is a map from $\mathcal A$ to the set of 
nonempty finite words on $\mathcal A$. We use concatenation to 
extend $\theta$ to a map on finite and infinite words 
from $\mathcal A$. We say that $\theta$ is {\em primitive} if there is some 
$k\in \N$ such that for any $a,a'\in \mathcal A$,
the word $\theta^k(a)$ contains at least one  occurrence of $a'$. 
By iterating $\theta$ on any fixed letter in $\mathcal A$, 
we obtain  one-sided right-infinite points $u=u_0 \ldots $  
such that $\theta^j(u)=u$ 
for some natural $j$. Similarly we can obtain one-sided left-infinite points $v= \ldots v_{-1}$ such that  $\theta^j(v)=v$. A bi-infinite periodic point for $\theta$ is a concatenation of a left-infinite periodic point $v= \ldots v_{-1}$ and a right-infinite periodic point  $u=u_0 \ldots $ provided that $v_{-1}u_0$ appears in some word $\theta^k(a)$.  
The pigeonhole principle implies that  $\theta$-periodic points
always exist, and, for primitive substitutions, we define 
$ X_\theta$ to be the shift orbit closure of any one of 
these bi-infinite $\theta$-periodic points and call 
$( X_\theta,  \sigma)$ a {\em  substitution shift}.
The substitution $\theta$ has
{\em (constant) length~r}  if for each $a\in \mathcal A$, 
$\theta (a)$ is a word of length $r$.

Let  $\theta$ be  a primitive length $r$ substitution such that  $X_\theta$ is infinite. We say that $\theta$ is  {\em recognizable} if any  $x \in X_\theta$ 
can be written in a unique way as 
$x=\sigma^k(\theta(y))$ where $y\in X_\theta$ and $0\leq k<r$. Our substitutions are recognizable,  see \cite{host} for a proof in the 
injective case, which is all we will need.

Let $\theta$ be a length $r$ substitution. We write 
\[\theta (a)= \theta_0(a) \ldots \theta_{r-1}(a);\]
with this notation we 
see that for each  $0\leq i \leq r-1$, we have a map
$\theta_i:\mathcal A \rightarrow \mathcal A$ where $\theta_i(a)$ is the 
$(i+1)$-st letter of $\theta(a)$. Similarly, for any $k$ we  write 
$(\theta^k)_i$ to denote the map which sends  $a\in \mathcal A$ to  the $(i+1)$-st letter of $\theta^k(a)$.

Let $\theta$ be a primitive length $r$ substitution with fixed 
point $u$, and with $(X_\theta,\sigma)$  infinite. 
The {\em height }$h=h(\theta)$ of $\theta$   is defined as 
\[
h(\theta):= \max \{n\geq 1: \gcd(n,r)=1, n | \gcd\{a: u_a=u_0 \} \}\, .
\]
If $h>1$, this means that $\mathcal A$ decomposes into $h$ disjoint 
subsets: $\mathcal A_1\cup\ldots\cup \mathcal A_h$, where a symbol from
$\mathcal A_i$ is always followed by a symbol from $\mathcal A_{i+1}$.

If  $\theta$ has a nontrivial height $h\geq 2$, then the shift $(X_\theta,\sigma)$ is topologically conjugate to a  constant height suspension of a height one substitution shift $(X_{\theta'}, \sigma)$, where $\theta'$ is called the {\em pure base} of $\theta$. In this case
$(X_{\theta}, \sigma) 
\cong(X_{\theta'}\times \{0,\ldots,h-1\},T)$ where

\begin{equation*}
 T(x,i):=
\begin{cases}
  (x,i+1)      & \text{ if  } 
0\leq  i<h-1    \\
    (\sigma'(x), 0)   & \text{ if } i=h-1
\end{cases}
\end{equation*} 
We refer the reader to \cite[Remark 9, Lemmas 17 and 19]{dekking} for details.

For every  $\Psi \in \Aut(X_{\theta'},\sigma')$ and  $0\leq i\leq h-1$
we can define $\Psi_i \in \Aut(X_{\theta'} \times \{0,\ldots,h-1\}, T)$ as

\begin{equation*}
\Psi_i(x,j) :=
\begin{cases}
  (\Psi(x), j+i) & \text{ if  } 
   j+i<h \\
       (\Psi(\sigma'(x)), j+i \mod h)     & \text{ if } j+i\geq h.
\end{cases}
\end{equation*}
The following is proved in \cite[Proposition 3.5]{coven-quas-yassawi}. 

\begin{prop}\label{prop:automorphism_with_height}
Let $\theta$  be a primitive,  length $r$ substitution, of height $h$, and  
such that $(X_{\theta},\sigma)$ is infinite. Let $\theta'$ 
be a pure base of $\theta$.
Then $ \Aut(X_{\theta},\sigma) =
\{   \Psi_j: \Psi \in \Aut(X_{\theta'},\sigma') \mbox{ and }  0\leq j<h  \}.$ 
\end{prop}

Proposition \ref{prop:automorphism_with_height} tells us that we need only study automorphism groups of shifts $(X_\theta,\sigma)$ where $\theta$ has height one, and henceforth this will almost always be a standing assumption.

\subsubsection{The maximal equicontinuous factor of a constant length substitution shift}

The {\em maximal equicontinuous factor} of a dynamical system $(X,\sigma)$ encodes its continuous eigenvalues. For a precise definition, see \cite{coven-quas-yassawi}.
The following is shown by Dekking \cite{dekking}, with partial 
results by Kamae \cite{kamae} and Martin \cite{martin}. 
\begin{thm} \label{thm:dekking}
Let $\theta$  be a primitive,  length $r$ substitution, of height $h$, 
and such that $(X_\theta,\sigma)$ is infinite. Then 
the maximal equicontinuous factor of $(X_\theta, \sigma)$ is $(\Z_r \times \Z/h\Z , +(1,1))$.
\end{thm}

\subsubsection{ Automatic shifts and sliding block code representations}

Let $\theta:\mathcal A\rightarrow \mathcal A^r$ be a primitive constant length substitution. If $\tau:\mathcal A \rightarrow \mathcal B$ and  $Y:= \tau(X_\theta)$, then $Y$ is closed and  $\sigma$-invariant, and we call $(Y,\sigma)$ an {\em $r$-automatic shift},  specifying that it is generated by $(\theta, \tau)$ when necessary. We use this terminology since for $u$ one-sided and $\theta$-fixed, $\tau(u)$ is known as an  {\em $r$-automatic} sequence. Note that $(Y,\sigma)$ is a factor of $(X_\theta,\sigma)$ via a radius zero factor map.

It can happen that $X_\theta$ is infinite, but $\tau(X_\theta)$ is finite. We are not interested in this case and will always require that the automatic shifts we consider are infinite. If $(X_\theta,\sigma)$ is minimal and $Y=\tau(X_\theta)$, then $(Y,\sigma)$ is minimal.

 Given a constant length substitution $\theta:\mathcal A\rightarrow \mathcal A^{r}$,  $\ell\in \N$ and $0\leq k <r$, we define $\theta^{(\ell,k)}:\mathcal A^{\ell}\rightarrow (\mathcal A^{(\ell)})^{r}$, the {\em $k$-shifted $\ell$-sliding block representation of  $\theta$}, which appears in \cite[Section 5.4]{Queffelec}. Namely if $(\alpha_1\ldots \alpha_\ell)\in \mathcal A^{\ell}$,  and $\theta (\alpha_1\ldots \alpha_\ell)= a_1 \ldots a_{\ell r}$, let
\[ \theta^{(\ell)}( \alpha_1\ldots \alpha_\ell    ) := (a_{k+1}\ldots a_{k+\ell})(a_{k+2} \ldots a_{k+\ell +1})\ldots (a_{k+r} \ldots a_{k+ r+\ell-1}).    \]
If $k=0$ we will write $\theta^{(\ell)}$ for $\theta^{(\ell, 0)}$.
It is straightforward to show that for each $\ell\geq 1$ and each $0\leq k <r$, $(X_{\theta^{(\ell,k)}}, \sigma)$ is topologically conjugate to  $(X_\theta, \sigma)$.
 The $\ell$-sliding block representation is particularly useful as it allows us to change the radius of a factor map. 

\begin{lem}\label{sliding-block-rep}

Let $(Y,\sigma)$ be a shift and let $\Phi:X_{\theta}\rightarrow Y$ be a  shift-commuting continuous map with left radius $\ell$ and right radius $r$. Then there is a shift commuting map $\Phi^{(\ell+r+1)}:X_{\theta^{(\ell+r+1)}}\rightarrow Y$, with left and right radius zero.
\end{lem}

\subsection{Maximal equicontinuous factors of automatic shifts}

Let $(Y,\sigma)$ be an automatic shift generated by $(\theta,\tau)$. 
As $(Y,\sigma)$ is a topological factor of $(X_\theta, \sigma)$, its maximal equicontinuous factor must be a factor of that of $(X_\theta, \sigma)$, which is $(\Z_r \times \Z/h\Z, +(1,1))$.
  In principle, $(\Z_r, +1)$ may have smaller non-finite factors, in particular when $r$ is not prime. In the next lemmas we show that as long as $(Y, \sigma)$ is not finite  it always has $(\Z_r, +1)$ as an equicontinuous factor. We briefly describe the required set up, referring the reader to \cite{dekking}  and \cite{Gottschalk-Hedlund} for more detail. Recall that a  partition $\{ P_1, \ldots , P_k\}$ of a shift $(X,\sigma)$ is {\em cyclic} if 
$\sigma({P_i}) = P_{i+1\bmod k}$. A partition is {\em $\sigma^{n}$-minimal} if each of its elements are $\sigma^{n}$-minimal, i.e. each $(P_i, \sigma^n)$ is a minimal shift. For each $n$, any minimal shift has a cyclic $\sigma^n$-minimal cyclic partition, which by minimality is unique up to cyclic permutation of its members. Let $\gamma(n)$ be the cardinality of the  cyclic $\sigma^{n}$-minimal partition.
In the case that we have a substitutional shift $(X_\theta,\sigma)$, Dekking \cite[Lemma II.7]{dekking} shows that  $\gamma(r^n) =r^n$, and that if 
$P_0:=\theta^n(X_\theta)$, then $P_0$ generates 
a $\sigma^{r^n}$-cyclic partition.

The proof of the following lemma is very similar to that of  \cite[Lemma 7]{dekking}.

\begin{lem}\label{good-skeleton}
If $(Y, \sigma)$ is an infinite primitive $r$-automatic shift, then $\gamma(r^n)=r^n$.
\end{lem}

\begin{proof}

First we show that there are at most $C+C^2(\gamma(r^n)-1)$ words of length $r^n$ in the language $\mathcal L_Y$ defined by $Y$. If $Y=\tau(X_\theta)$, with $\theta$ defined on $\mathcal A$ of cardinality $C$, let $v=\tau(u)$ for any $u$ bi-infinite and $\theta$-fixed. Now $v\in \tau(\theta^n(X_\theta))$ for each $n$. Therefore $\sigma^{k\gamma(r^n)}(v) \in\tau(\theta^n(X_\theta))$, so that $v$ is composed of overlapping blocks of the form $\theta^n(a)$ spaced at intervals $\gamma(r^n)$. As there are at most $C$ words $\theta^n(a)$, and  at most $C^2$ words $\theta^n(a)\theta^n(b)$, this implies that there are at most  $C+C^2(\gamma(r^n)-1)$ words of length $r^n$ in $\mathcal L_Y$. Now if $\gamma(r^n)<r^n$, then by \cite[Lemma 3(vi)]{dekking}, we have $\frac{\gamma(r^n)}{r^n}\rightarrow 0$. This implies that for $n$ large, $\mathcal L_Y$ contains fewer than $r^n$ words of length $r^n$, contradicting the assumption that $Y$ is infinite.
\end{proof}
We note that for each $k$, $\tau(\theta^k(X_\theta))$ is a closed, minimal  $\sigma^{ r^k   }$-invariant set. This follows from the fact $\theta^k(X_\theta)$ satisfies these properties. Also,  $\tau(\theta^k(X_\theta))$ must generate a $\sigma^{ r^k   }$-invariant partition. 

We fix the maximal equicontinuous factor  map that we consider henceforth. First suppose we consider a substitutional shift $(X_\theta,\sigma)$. 
 We  write $\Lambda_{r^n}(x)=i$ if $x\in \sigma^{i} \theta^{n}(X_\theta)$. 
Note that if  $\Lambda_n(x) = i$, then $\Lambda_{n+1}(x) \equiv i \mod r^{n}$.
Using this we can define $\pi(x):= \ldots x_2\,x_1\,x_0$ where for each 
$n\in \N$, $\Lambda_n(x)=\sum_{i=0}^{n-1} r^{i}x_i$.  

If $\theta$ has trivial height, we have defined a maximal equicontinuous factor map.  If $\theta$ has positive height $h$, then $X_\theta$ also has a $\sigma^h$-cyclic partition.
Here we choose an arbitrary but fixed  base, and use it to extend $\pi$ to a maximal equicontinuous factor map.

Finally if we consider an automatic shift $(Y,\sigma)$ generated by $\theta$, the convention we use is that we take the image of these partitions by $\tau$.

 Now by Lemma \cite[Lemma 11]{dekking}, $(X_\theta,\sigma) $ has no continuous irrational  eigenvalues, and so neither does $(Y,\sigma)$. Thus the maximal
equicontinuous factor of $(Y,\sigma)$ is generated by the  $\sigma^{n}$-minimal cyclic partitions \cite{ellis-gottschalk-1960}. On the one hand, the maximal
equicontinuous factor of $(Y,\sigma)$ must be a factor of $(\Z_r \times \Z/h\Z, +(1,1))$. On the other hand Lemma \ref{good-skeleton} tells us that each $(\Z/r^n\Z,+1)$ is a factor of 
$(Y, \sigma)$. We have proved

\begin{thm}\label{skeleton}
If $(Y, \sigma)$ is an infinite $r$-automatic shift generated by $(\theta,\tau)$ where $\theta$ is primitive and of height $h$, then $(\Z_r \times \Z/\bar{h}\Z,+(1,1))$ is the maximal equicontinuous factor of $(Y,\sigma)$ for some $\bar h$ dividing $h$. Furthermore, we can take the maximal equicontinuous factor maps $\pi_X$ and $\pi_Y$ to satisfy  $\pi_X=\pi_Y\circ \tau$.
\end{thm}

\subsection{Automorphisms of constant length substitution shifts}

We recall  key ingredients that will be essential in our study of automorphism groups of automatic shifts.

Let $\End(X,\sigma)$ be the set of endomorphisms of $(X,\sigma)$.   
Given two shifts $(X,\sigma)$  and $(Y,\sigma)$, and let $\Fac(X,Y)$ denote the set
of factor maps from $(X,\sigma)$  to $(Y,\sigma)$. The following is proved in \cite[Theorem 3.3]{coven-quas-yassawi}.

\begin{thm}\label{thm:ethan}  
Let  $(X,\sigma)$  and $(Y,\sigma)$ be  infinite minimal 
shifts.
Suppose that the group rotation $(G,R)$ is the maximal 
equicontinuous factor of both $(X,\sigma)$ and $(Y,\sigma)$ and let
$\pi_X$ and $\pi_Y$ be the respective factor maps.
Then there is a map $\kappa:  \Fac(X,Y) \rightarrow G$
such that \[\pi_Y(\Phi(x))=\kappa(\Phi)+\pi_X(x) \] for all 
$x\in X$ and $\Phi  \in \Fac(X,Y)$. Also 
\begin{enumerate}
\item
 if $(Z,\sigma)$ is another 
shift which satisfies the assumptions on $(X,\sigma)$, then
$\kappa(\Psi\circ\Phi)=\kappa(\Psi)+\kappa(\Phi)$ for 
$\Phi \in \Fac (X,Y)$,  $\Psi \in
\Fac(Y,Z)$, and 
\item if $\min_{g\in G}|\pi_X^{-1}(g)|=\min_{g\in G}|\pi_Y^{-1}(g)|=c<\infty$,
then
$\kappa$ is at most  $c$-to-one. In particular, if $\pi_X$ and 
$\pi_Y$ are somewhere one-to-one, then $\kappa $ is an injection.
\end{enumerate}
\end{thm}

A constant-length substitution is called \emph{injective} if $\theta(i)\ne
\theta(j)$ for any distinct $i$ and $j$ in $\mathcal A$. 
It will be convenient to deal with injective substitutions in what follows. 
The following theorem of Blanchard, Durand and Maass allows us to restrict 
our attention to that situation.

\begin{thm}[\cite{BDM}]\label{thm:injective}
Let $\theta$ be a constant-length substitution such that $X_\theta$ is 
infinite. Then there exists an injective constant length substitution $\theta'$
such that $(X_{\theta},\sigma)$ is topologically conjugate to 
$(X_{\theta'},\sigma)$. 
\end{thm}

We say that $\theta$ {\em has column number $c$ } if 
\begin{align}\label{column number}
	c = c(\theta) := \min_{ 0\leq j<r^k,\, k\in \N} \abs{\theta^k(\mathcal{A})_j}.
\end{align}
In other words,
  for some $k\in \N$, 
and some $(i_1, \ldots ,i_k)$,  
$|\theta_{i_1}\circ \ldots \circ\theta_{i_k}(\mathcal A)|=c$ and $c$ 
is the least such number.

The following proposition, \cite[Proposition 3.10]{coven-quas-yassawi} tells us that up to a shift, 
every automorphism has a small radius.

\begin{prop}\label{prop:radius_two} 
Let $\theta$  be an injective primitive, length $r$ substitution  of height one, 
and  such that $(X_\theta,\sigma)$ is infinite. 
If $\Phi\in \Aut (X_\theta,\sigma)$ has the property
that $\kappa(\Phi)\in\Z_r\setminus \Z$ is periodic, then $\Phi$ 
has right radius zero and left radius at most 1. 

Further, if $\kappa(\Phi)=k/(1-r^p)$ for $0<k<r^p-1$, one has
\begin{equation}\label{eq:commuting-relation}
\theta^{-c!p }\circ \sigma^{-k(1+r^p+r^{2p}+\ldots + r^{(c!-1)p})}\circ 
\Phi\circ \theta^{c!p} = \Phi,
\end{equation}
where $c$ is the column number.

If $\kappa(\Phi)=0$, then $\Phi$ has left and right radius at most 1 and 
$\Phi$ satisfies 
\begin{align} \label{eq:commuting-relation-2}
	\theta^{-c!} \circ \Phi \circ \theta^{c!} = \Phi.
\end{align}

\end{prop}

\begin{example}\label{radius-can-be-large} In this example we show that for automatic shifts, it is no longer the case that automorphisms with a zero $\kappa$-value have radius at most one,
as in Proposition \ref{prop:radius_two}. Consider the substitution
 $\theta$
\begin{align*}
  \theta(a) &= ac\\
  \theta(b) &= bd\\
  \theta(c) &= ab\\
  \theta(d) &= ba.
\end{align*}
Now  $X_{\theta}$  has an automorphism $\Phi$ whose $\kappa$-value equals zero,  given by the local rule $f$ defined by the permutation $(ab)(cd)$.
Next we consider the coding $\tau: \{a,b,c,d\} \to \{x,y,z\}$ defined by
\[
  \tau(a) = x, \,\,\,\,\,
  \tau(b) = y, \,\,\,\,\,
  \tau(c) = x \,\,\mbox{ and } \,\,
  \tau(d) = z,
\]
and  let $Y$ be the automatic shift defined by $(\theta,\tau)$. One can verify  that $\tau$ is injective, so 
there exists an automorphism of $Y$ that corresponds to $\Phi$, whose local rule is given by $g$, partially given by
\begin{align*}
  g(y) &= x\\
  g(z) &= x.
\end{align*}
However, the local rule is a bit more complicated when considering the image of $x$ under $g$, as we need to consider the preceeding and following $2$ elements, and the rule is given in 
Table \ref{local-rule-table-large-radius}.
This table contains all appearing blocks of length $5$ such that the middle element is $x$. To be able to define this local rule, it is crucial that every appearence of such a block uniquely defines the preimage of the middle element.

\begin{table}[]
\begin{tabular}{|l||l|l|l|}
\hline
     & x & y & z \\ \hhline{|=#=|=|=|}
xxxx & y & z &   \\ \hline
xxxy & y & y &   \\ \hline
xyxx & y & y &   \\ \hline
yxxx &   & z &   \\ \hline
yxxy &   &   & z \\ \hline
yzxx & y & y &   \\ \hline
zxxx & z & z &   \\ \hline
zxxy &   &   & z \\ \hline
zyxx & y &   &   \\ \hline
zyxy &   &   & y \\ \hline
\end{tabular}
\vspace{0.5cm}
 \caption{The local rule for the factor map in Example \ref{radius-can-be-large}.}
\label{local-rule-table-large-radius}
\end{table}

\end{example}

The next result, \cite[Proposition 3.21]{coven-quas-yassawi} gives us conditions to check whether a local rule defines an element of $\Aut(X_\theta,\sigma)$.
If $\Phi$ has local rule $f$, left radius $\ell$ and right radius $r$,  we write $\Phi = \Phi^f_{\ell,r}$.
\begin{prop}\label{prop:last_nail_lemma}
Let $\theta$  be an injective primitive, length $r$ substitution on 
$\mathcal A$,  with column number $c$, of height one, and  
such that $(X_\theta,\sigma)$ 
is infinite. 
Let $p\ge 1$ and $0<k<r^p-1$. 

Let $f\colon \mathcal A^2\to \mathcal A$.
Then $\Phi:=\Phi^f_{1,0}$ satisfies $\Phi\in \Aut(X_\theta,\sigma)$  
and $\kappa(\Phi)=k/(1-r^p)$ if and only if:
\begin{enumerate}
\item 
if $x_0x_1x_2\in\mathcal L_{X_\theta}$, then $f(x_0,x_1)f(x_1,x_2)\in
\mathcal L_{X_\theta}$; and \label{cond:fblock}
\item For each $x_{-1}x_0x_1\in\mathcal L(X_\theta)$ and $0\le i<r^{c!p}$
\label{cond:crelf}
$$
f(\theta^{c!p}(x_{-1}\cdot x_0)_{i-1},\theta^{c!p}(x_{-1}\cdot x_0)_i)
=\theta^{c!p}(f(x_{-1},x_0)f(x_0,x_1))_{N+i},
$$
where $N=k(1+r^p+\ldots+r^{(c!-1)p})$ and $x_{-1}\cdot x_0$ denotes a 
finite segment of a bi-infinite sequence with $x_{-1}$ in the $-1$st coordinate
and $x_0$ in the $0$-th coordinate.
\end{enumerate}
Similarly given $g\colon \mathcal A^3\to \mathcal A$, $\Phi^g_{1,1}$ 
satisfies $\Phi^g_{1,1}\in\Aut(X_\theta,\sigma)$
and $\kappa(\Phi^g_{1,1})=0$ if and only if 
\begin{enumerate}
\setcounter{enumi}{2}
\item $g(x_0,x_1,x_2)g(x_1,x_2,x_3)\in\mathcal L_{X_\theta}$
whenever $x_0x_1x_2x_3\in\mathcal L_{X_\theta}$; and
\label{cond:gblock}
\item For any $x_{-1}x_0x_1\in\mathcal L_{X_\theta}$ and $0\le i<r^{c!}$, 
\label{cond:crelg}
\begin{equation*}
g(u_{i-1},u_i,u_{i+1})=
(\theta^{c!}g(x_{-1},x_0,x_1))_i \, ,
\end{equation*}
where $u=u_{-r^{c!}}u_{-r^{c!}+1}\ldots u_{-1}\cdot u_0\ldots
u_{2r^{c!}-1}$ is the word 
obtained by applying $\theta^{c!}$ to the word $x_{-1}\cdot x_0x_1$.
\end{enumerate}
\end{prop}

\begin{remark}
Proposition \ref{prop:radius_two}  implies that if $\Phi$ is a (topological) automorphism, then it is also a measure preserving (measurable) automorphism of 
$(X_\theta, \sigma,\mu)$, where $\mu$  is the unique $\sigma$-invariant measure on $X_\theta$. Conversely, Host and Parreau \cite{HP} show that if $\Phi$ is a measure preserving (measurable) automorphism of $(X_\theta, \sigma,\mu)$, with the additional conditions that $\theta$ has column number at least two, and that $\theta$ is {\em reduced}, then $\Phi$ is almost everywhere defined by a sliding block code of radius one, and that it satisfies a version of \eqref{eq:commuting-relation} or~\eqref{eq:commuting-relation-2}.    Now Proposition \ref{prop:last_nail_lemma} tells us that $\Phi$ is therefore (almost everywhere equal to) a topological automorphism. We note that reduced substitutions (defined in \cite{HP}) are always strongly injective, a notion we define and study in Section \ref{strongly-injective-section}.
\end{remark}

\section{Topological factors of substitution shifts}\label{automatic-section}

In this section we study automorphism groups of factors of constant-length substitution dynamical systems. We show in particular that we can both produce new automorphisms, and lose automorphisms, when we pass to a factor. Moreover, we show that any shift factor is indeed topologically conjugate to a constant length substitution dynamical system, via a code, i.e. a letter-to letter map, i.e. it is an automatic shift, so that we can explicitly describe its automorphism group. In the last part we consider the special case when the factor map is uniform-to-$1$.

\subsection{Automorphism groups and factors}
Our motivation for the main results in Section \ref{automatic-section} arose as we noticed that in general there is no relationship between the automorphism group of a substitution shift $(X_{\theta}, \sigma)$ and that of one of its topological factors $(Y, \sigma) = (\tau(X_{\theta}), \sigma)$. We summarise this with the following two examples.

\begin{example}{\bf  An example where ${\bf \Aut(X_\theta,\sigma)\backslash \Aut(Y,\sigma)}$ is nonempty.}
Define $\theta$ as
\begin{align*}
  \theta(a) &= ad\\
  \theta(b) &= bc\\
  \theta(c) &= ab\\
  \theta(d) &= ba.
\end{align*}
This substitution has column number 2 and the fixed points, namely $b.a$, $c.a$, $a.b$, $d.b$, comprise  the only  $\pi$-fibre with more than two elements, where $\pi:X_\theta\rightarrow \Z_2$ is the maximal equicontinuous factor map. Now $X_\theta$ has a nontrivial automorphism $\Phi\in \Aut(X_\theta,\sigma)$ corresponding to the local rule given by the letter-exchanging map $(ab)(cd)$. Note that $\Phi$ has $\kappa$-value zero, and it acts by permuting elements within the fibres of $X_\theta$.

Now consider the projection $\tau: \{a,b,c,d\} \to \{x,y\}$ given by
$  \tau(a) =\tau(b) =  x$, 
  $\tau(c) = \tau(d) =y$.
 We find that $Y = X_{\bar{\theta}}$ where
\begin{align*}
  \bar{\theta}(x) &= xy\\
  \bar{\theta}(y) &= xx,
\end{align*}
where $\theta$ has two fixed points $x.x$ and $y.x$.  Since $\theta$ has a coincidence, i.e. $c(\theta) = 1$, it does not have any nontrivial automorphisms with a zero $\kappa$-value. In other words, The map $\tau$ is everywhere two-to-one, and it collapses the fibres for $X_\theta$.  
Note that if $\Phi$ is sent to an automorphism via $\tau$,  it must be an automorphism which also permutes within fibres. There is no non-trivial $\Phi' \in \Aut(Y,\sigma)$ such that $\kappa(\Phi') = 0$ as $\bar{\theta}$ has a coincidence.  In other words, here both the identity map and $\Phi$ are collapsed by $\tau$ to the identity.
\end{example}

It is appropriate to put the next example here, even though our claims use ideas from Section \ref{strongly-injective-section}.
\begin{example}{\bf An example where ${\bf  \Aut(Y,\sigma)\backslash \Aut(X_\theta,\sigma)}$ is nonempty.}
We consider the following primitive substitution $\theta$

\begin{align*}
  \theta(a) &= aac\\
  \theta(b) &= bca\\
  \theta(c) &= bba.
\end{align*}
Note that this substitution is {\em strongly injective} (Definition \ref{strongly}), and so if there exists $\Phi \in \Aut(X_{\theta})$ with $\kappa(\Phi) = 0$ then  by Theorem~\ref{characterisation-of-kernel}, it must have radius $0$, i.e. it must exchange letters. Now the fixed points of $\theta$ are $a.a,c.a,a.b,c.b$ and so if there is a non-trivial automorphism it must send $a$ to $b$ and $a$ to $c$, a contradiction. Thus $\Phi=\Id$.
However, now we consider the coding $\tau(a)=x$, $\tau(b)=\tau(c)=y$,
and we find that $Y = X_{\bar{\theta}}$ where 
\begin{align*}
  \bar{\theta}(x) &= xxy\\
  \bar{\theta}(y) &= yyx.
\end{align*}
One finds easily that there exists a nontrivial $\Phi \in \Aut(Y,\sigma)$ with $\kappa$-value zero, namely $(xy)$.

We remark here that  $\Phi\in \Aut(Y,\sigma)$ does define a measurable  automorphism $F$ of $(X_\theta,\sigma)$. In particular the only points where $F$ is not defined are on  the orbits of the fixed points, and also the orbit of the points in $\pi^{-1}(\ldots 111)=\pi^{-1}(-\frac{1}{2})$, with $\pi:X_\theta\rightarrow \Z_3$ the maximal equicontinuous factor map. The map $F$ is not defined almost everywhere by a radius one sliding block code (or by any finite sliding block code, for that matter), which highlights that the result of Host and Parreau~\cite{HP} is not true for not-reduced substitutions.
\end{example}

\subsection{Automatic shifts as substitution shifts}

In this section we prove that any shift factor of a primitive constant length substitution shift is  topologically conjugate, via a radius zero code, to a constant length substitution shift.  Note that by Lemma \ref{sliding-block-rep}, we can assume that the factor is automatic to begin with.
\begin{definition}
  We call $a \neq b \in \mathcal{A}$ a \emph{periodic pair} if there exists $p=p(a,b)$ and $j< r^{p}$ such that $\theta^{p}(a)_j = a$ and $\theta^{p}(b)_j = b$.
Note that this is equivalent to the requirement that
  for all $k > 0$ there exists $j< r^{k  p}$ such that $\theta^{k  p}(a)_j = a$ and $\theta^{k  p}(b)_j = b$. 
  We define 
  \begin{align*}
    p(\theta) = lcm\{p(a,b): a,b \text{ are a periodic pair}\},
  \end{align*}
  and call a substitution $\theta$ with $p(\theta) = 1$ \emph{pair-aperiodic}.
\end{definition}

Readers who are familiar with techniques used in substitutions will be aware that one often replaces  a substitution $\theta$ by an appropriate power so that, without loss of generality, one can assume that all $\theta$-periodic points are $\theta$-fixed.  In a similar spirit, we show in the following lemma that  without loss of generality 
  we can work with pair-aperiodic substitutions.
\begin{lem}\label{pair-aperiodic}
  Let $\theta$ be a substitution of constant length. Then $\theta^{p(\theta)}$ is pair-aperiodic.
\end{lem}
\begin{proof}
  A simple computation shows that if $a,b$ are a periodic pair for $\theta$ with period $p_1$, then they are also a periodic pair for $\theta^n$ with period $p_2 = p_1/\gcd(p_1,n)$.
  Furthermore, one finds directly that any periodic pair for $\theta^n$ is also a periodic pair for $\theta$.
  The result follows.
\end{proof}

\begin{definition}
  We call $a,b \in \mathcal{A}$ an \em{asymptotic disjoint pair} if for all $k > 0$ there exists 
  $j_k<r^k$ such that $\theta^k(a)_{j_k} \neq \theta^k(b)_{j_k}$.
\end{definition}
\begin{definition}
We call a pair $(c,d) \in \mathcal{A}^2$ $k$-reachable  from $(a,b) \in \mathcal{A}^2$ if there exists $j<r^k$ such that $\theta^k(a)_j = c$ and $\theta^k(b)_j = d$.

\end{definition}

  It is straightforward to check that reachability is transitive, i.e. if $(b_1,b_2)$ is $k_1$-reachable from $(a_1,a_2)$ and $(c_1,c_2)$ is $k_2$-reachable from $(b_1,b_2)$, then $(c_1,c_2)$ is $(k_1+k_2)$-reachable from $(a_1,a_2)$.

\begin{lem}\label{le:disjoint}
  Let $\theta$ be a pair-aperiodic substitution of constant length and $(a,b)$ an asymptotic disjoint pair. 
   Then there exists a periodic pair $(a',b')$ such that for any $k \geq 2|\mathcal{A}|^2$, and any $(c,d)$ that is reachable from $(a',b')$, it is also $k$-reachable from $(a,b)$. 
   \end{lem}
\begin{proof}
  As $(a,b)$ is an asymptotic disjoint pair, there exists for any $k \in \N$ some $j_k$ such that $\theta^k(a)_{j_k} \neq \theta^k(b)_{j_k}$. We consider the finite sequence  $(\theta^k(a)_{j_k},\theta^k(b)_{j_k})_{k=0,\ldots,|\mathcal{A}|^2}$. We note that we can choose the $j_k$ such that $j_k = \lfloor j_{k+1}/r \rfloor$.
 (For example one first chooses $j_{|\mathcal A|^{2}}$.)

  As every element in this sequence belongs to $\mathcal{A}^2$ and the sequence has $|\mathcal{A}|^2 +1$ elements, there exists $(a',b')$ and $0\leq k_1 < k_2 \leq |\mathcal{A}|^2$ such that
  $\theta^{k_i}(a)_{j_{k_i}} = a'$ and $\theta^{k_i}(b)_{j_{k_i}} = b'$ for $i = 1,2$. We claim that  $(a',b')$ are a periodic pair. For,
  we know that $j_{k_2} = j_{k_1} r^{k_2-k_1} + j'$, which gives
  \begin{align*}
    a' &= \theta^{k_2}(a)_{j_{k_2}} = \theta^{k_1 + (k_2 -k_1)}(a)_{j_{k_1} r^{k_2-k_1} + j'}  = \theta^{k_2-k_1}(\theta^{k_1}(a)_{j_{k_1}})_{j'}\\
    & = \theta^{k_2 - k_1}(a')_{j'},
  \end{align*}
  and the analogous statement for $b, b'$, and the claim follows.
   Since $\theta$ is pair-aperiodic, $(a',b')$ is $1$-reachable from $(a',b')$.
  By the transitivity of reachability it follows that $(a',b')$ is reachable from $(a,b)$ by $k$ steps for any $k \geq |\mathcal{A}|^2$.
  Finally, if $(c,d)$ is reachable from $(a',b')$, then by the pigeonhole principle,  it must be $k'$- reachable for some $k' \leq |\mathcal{A}|^2$.  The result follows by transitivity.
\end{proof}

In this last proof, we used pair-aperioidicty strongly. We use it in a similar fashion in the next lemma.
\begin{lem}\label{le:not_disjoint} 
Let
 $\theta$  be a primitive, pair-aperiodic substitution. If $(a,b)$ is not an asymptotic disjoint pair, then $\theta^{|\mathcal{A}|^2}(a) = \theta^{|\mathcal{A}|^2}(b)$.
\end{lem}

\begin{proof}
Since $(a,b)$ is not asymptotic disjoint, there exists a minimal $K$ such that $\theta^K(a)=\theta^K(b)$. Suppose that $K> |\mathcal A|^2$. Then for each $0\leq i\leq K-1$ there is a $j_i$ such that $\alpha_i:=(\theta^{i}(a))_{j_i}\neq \beta_i :=(\theta^{i}(a))_{j_i}$  
and where if $i< \bar i$ then $(\alpha_{\bar i},\beta_{\bar i})$ is reachable from $(\alpha_i,\beta_i)$. Since $K-1\geq  |\mathcal A|^2$, we have $(\alpha_i,\beta_i)= (\alpha_{\bar i},\beta_{\bar i})$ for some $i< \bar i$. Thus $(\alpha_i,\beta_i)$ is a periodic pair. But now since $\theta$ is pair aperiodic, for each $n$ $(\alpha_i,\beta_i)$ appears as $(\theta^n(a))_{k_n}=(\theta^n(b))_{k_n}$ for some $k_n$. This contradicts the fact that $\theta^K(a)=\theta^K(b)$.
\end{proof}

\begin{definition}
Let $(\theta,\tau)$ be an $r$-automatic pair.
  We call two letters $a\neq b $ \emph{indistinguishable} (by $(\theta,\tau)$ ) if for any $k \geq 0$ and  $0\leq j< r^k$ we have $\tau(\theta^{k}(a)_j) = \tau(\theta^k(b)_j)$.
  We call the pair $(\theta, \tau)$ {\em  minimal} if every pair of letters is distinguishable by $(\theta,\tau)$.
\end{definition}

 Note that the letters $a, b$ are indistinguishable if and only if  for any $(c,d)$ that is reachable from $(a,b)$ (which includes $(a,b)$ by $0$ steps) we have $\tau(c) = \tau(d)$.

In the next lemma, a fundamental result from automata theory enables us to replace an $r$-automatic pair by a minimal $r$-automatic pair.

\begin{lem}\label{minimal}
  Let  $(Y,\sigma)$ be an infinite $r$-automatic shift generated by $(\theta',\tau')$ where $\theta'$ is primitive and pair-aperiodic.  
  Then there exists a minimal $r$-automatic pair $(\theta,\tau)$  which generates  $(Y,\sigma)$, where $\theta$ is  primitive and pair-aperiodic.
  \end{lem}
\begin{proof}
	We define an equivalence relation $a \sim b$ if $a,b$ are indistinguishable.
	This equivalence relation is obviously compatible with $\theta'$, so we can define
	\begin{align*}
		\theta([a])_j := [\theta'(a)_j]
	\end{align*}
	and it is also compatible with $\tau$, so we set
	\begin{align*}
		\tau([a]) := \tau(a).
	\end{align*}
	We need to check that $\theta$ is primitive and pair-aperiodic.
	As $\theta'$ is primitive, we know that there exists some $n$ such that for all $a,b \in \mathcal{A}'$ we have that $(\theta')^n(a)$ contains $b$ as a subword. Let us take $[c],[d] \in \mathcal{A}$. We know that $(\theta')^n(c)$ contains $d$ as a subword and therefore $\theta^n([c])$ contains $[d]$ as a subword. In other words $\theta$ is primitive.\\
	Let us now assume that $\theta^k([a])_j = [a]$ and $\theta^k([b])_j = [b]$, i.e. $([a],[b])$ is a periodic pair.
	This implies that $(\theta')^{k}_j$ is a map from $[a]$ to itself and also from $[b]$ to itself.
	Thus a power must have a fixed-point:
	\begin{align*}
		((\theta')^k_j)^n(a_0) &= a_0\\
		((\theta')^k_j)^m(b_0) &= b_0
	\end{align*}
	for some $a_0 \in [a], b_0 \in [b]$.
	Therefore we have
	\begin{align*}
		((\theta')^k_j)^{nm}(a_0) &= a_0\\
		((\theta')^k_j)^{nm}(b_0) &= b_0
	\end{align*}
	which means that $(a_0, b_0)$ is a periodic pair for $\theta'$. As $\theta'$ is pair-aperiodic we know that there exists $j_0$ such that
	\begin{align*}
		\theta'(a_0)_{j_0} &= a_0\\
		\theta'(b_0)_{j_0} &= b_0
	\end{align*}
	and, therefore,
	\begin{align*}
		\theta([a])_{j_0} &= [a]\\
		\theta([b])_{j_0} &= [b].
	\end{align*}
	This shows that $\theta$ is pair-aperiodic.
	
	Furthermore, by definition
	\begin{align*}
		\tau'((\theta')^{k}(a)) = \tau(\theta^k([a]))
	\end{align*}
	
	which shows that the language of $(Y \sigma)$ is the same as the language of $(\tau(X_{\theta}), \sigma)$.  Since shifts are uniquely determined by their language, the result follows. 	
	\end{proof}

We can now prove the main result of this section.

\begin{thm}\label{th:main}
  Let  $(Y,\sigma)$ be an infinite $r$-automatic shift, generated by $(\theta',\tau')$ where $\theta'$ is primitive. Then there exists an $r$-automatic pair $(\theta, \tau)$ that generates $(Y,\sigma)$ and such that $\tau:   X_\theta \rightarrow Y$ is a bijection.
\end{thm}

\begin{proof}

By Lemmas \ref{pair-aperiodic} and  \ref{minimal}, we can assume that $(Y,\sigma)$ is generated by the minimal   $r$-automatic pair $(\theta, \tau)$ where $ \theta$ is pair-aperiodic.

First we suppose that there exist $x \neq \tilde x \in X_{\theta}$ such that $\tau(x) = \tau(\tilde x)$.
As $X_{\theta}$ is recognizable \cite{Mosse} we have $x', \tilde x'$ such that 
$x = \sigma^{\ell} \theta^{2|\mathcal{A}|^{2}}(x')$ and  $\tilde x = \sigma^{\tilde \ell} \theta^{2|\mathcal{A}|^{2}}(\tilde x')$.   Now Theorem \ref{skeleton} tells us 
 that $\ell=\tilde \ell$.
Thus, we have $x', \tilde x'$ such that $\theta^{2|\mathcal{A}|^{2}}(x') \neq \theta^{2|\mathcal{A}|^{2}}(\tilde x')$ and $\tau(\theta^{2|\mathcal{A}|^{2}}(x')) = \tau(\theta^{2|\mathcal{A}|^{2}}(\tilde x'))$.

We consider the pair $(x'_n, \tilde x'_n) \in \mathcal{A}^2$. Suppose that for some  $n \in \Z$, 
$(x'_n, \tilde x'_n)$ is an asymptotic disjoint pair.
By Lemma~\ref{le:disjoint} we know that there exists a periodic pair $(a',b')$ such that for any $(c,d)$ that is reachable from $(a',b')$, we have that $(c,d)$ is $(2\abs{\mathcal{A}}^{2})$-reachable from $(a,b)$.
 As $\tau(\theta^{2|\mathcal{A}|^{2}}(x'_n)) = \tau(\theta^{2|\mathcal{A}|^{2}}(\tilde x'_n))$, this gives that $\tau(c) = \tau(d)$ and therefore that $a'$ and $b'$ are indistinguishable. This contradicts our assumption that  $(\theta, \tau)$ is minimal.
 
Thus $(x'_n, \tilde x'_n)$ is not an asymptotic disjoint pair for any  $n \in \Z$.
By Lemma~\ref{le:not_disjoint} we know that $\theta^{2|\mathcal{A}|^{2}}(x'_n) = \theta^{2|\mathcal{A}|^{2}}(\tilde x'_n)$ for each $n\in \Z$, i.e. $x = \tilde x$, another contradiction. The result follows.

\end{proof}

The following example tells us that starting with a pair-aperiodic substitution before we minimise is necessary in Theorem \ref{th:main}.

\begin{example}
  We consider the substitution $\theta': \{a,b,c,d\} \to \{a,b,c,d\}^{3}$ and $\tau': \{a,b,c,d\} \to \{x,y,z\}$ where
  \begin{align*}
    \theta'(a) &= abd\\
    \theta'(b) &= aad\\
    \theta'(c) &= add\\
    \theta'(d) &= acd
  \end{align*}
  and 
  \[  \tau'(a) =   \tau'(c) = x, \,\,\,\,\,   \tau'(b) = y \,\, \mbox{ and }\,\,
  \tau'(d) = z.  \]
 
  We find that $(\theta', \tau')$ is minimal, as the only possible indistinguishable pair is $(a,c)$, but $\tau'(\theta'(a)_1) = y \neq \tau'(\theta'(c)_1) = z$.
  
   However $\tau': X_{\theta'} \to Y$ is not injective. For, let
  $x = \lim_{n\to \infty} \sigma^{(3^{2n}-1)/2} \theta'^{2n}(a) $ and let $\tilde  x = \lim_{n\to \infty} \sigma^{(3^{2n}-1)/2} \theta'^{2n}(c)$. One can verify that $x_n = \tilde x_n$ for all $n \neq 0$ and $x_0 = a, \tilde x_0 = c$. Thus $\tau'(x) = \tau'(\tilde x)$.

  The ``problem" here is that $(a,c)$ is a periodic pair of period $2$ and we need to actually consider $\theta'^2$ instead of $\theta'$.
  \begin{align*}
    \theta'^2(a) &= abdaadacd\\
    \theta'^2(b) &= abdabdacd\\
    \theta'^2(c) &= abdacdacd\\
    \theta'^2(d) &= abdaddacd.
  \end{align*}
  Here we see that $(a,c)$ is indeed indistinguishable, so that $(\theta'^2, \tau)$ is not minimal. In this case we minimise and we find the new substitution
  \begin{align*}
    \theta(x) &= xyzxxzxxz\\
    \theta(y) &= xyzxyzxxz\\
    \theta(z) &= xyzxzzxxz
  \end{align*}
  along with $\tau(a)=x, \tau(b)=y, \tau(d)=z$ gives us the correct representation.
 
\end{example}

\subsection{Automorphism groups and uniform-to-one extensions.}

In this section we investigate topological factors of minimal dynamical systems $(X,T)$ where $X$ is compact and metric, and where the factor map is uniform-to-one.

Let $(X,T)$ be a such a dynamical system. If there exists an automorphism $\Phi$ of $(X, T)$ with $\Phi^k=\Id$ for some minimal $k$, then we claim that we can define an everywhere $k$-to-one factor of $(X, T)$. First note that if $k$ is minimal, then for no $\ell<k$ do we have $\Phi^\ell(x)=x$ for some $x\in X$.  For, if that were the case, then $\Phi^\ell$ equals the identity on a dense subset of $X$. By minimality, $\Phi^\ell= \Id$ and this contradicts minimality of $k$.  Now we can  identify points $\{ x, \Phi(x), \ldots , \Phi^{k-1}(x) \}$ that are in the same orbit and this is a well defined $T$-commuting map on $X$.

This assertion extends to the measurable setting. Namely,  if there exists a measurbale ergodic  bijection $\Phi$ of $(X, T, \mu)$, then we can use ergodicity to assert that there is a minimal $n$ such that for $\mu$-almost all $x$, $\Phi^n(x)=x$ and $n$ is minimal for that $x$. Now we can proceed with that $n$ as above.

In particular, if $\theta$ has column number $c$ and  $\Phi^c=\Id$ with $c$ minimal, then the factor defined above is an almost everywhere one-to-one extension of its maximal equicontinuous factor, 
and this is a convenient topological factor for a dynamical system to have,
see for example \cite{Baake-Lenz}, \cite{L-M-2018}, or \cite{Staynova} for uses. 
It is not always the case that such a factor exists  \cite{Herning}. 
It is natural to investigate how closely related automorphisms and everywhere $k$-to-$1$ factors are. The following result shows that for the case $k=2$ we have a $1$-to-$1$ correspondence. We are grateful to T. Downarowicz and M. Lemanczyk who pointed this general result out to us via private communication~\cite{Downarowicz_Lemanczyk}.

\begin{thm}\label{automorphism-almost-automorphic-v2}
Let $(X,T)$  and $(Y,S)$ be  minimal topological dynamical systems  on  compact metric spaces $X$ and $Y$. If $\pi:(X,T)\rightarrow (Y,S)$ is a  uniformly 2-to-1 topological  factor map, then there exists $\Phi\in \Aut(X,T)$ such that $ \Phi^{2}=\Id$.
\end{thm}

\begin{proof}
Several times in this proof, compactness allows us to  drop to a subsequence; we will do this without re-indexing.

For $y\in Y$, write $\{ x,\tilde x \}= \pi^{-1}(y)$; define $\Phi(x)=\tilde x$ and $\Phi(\tilde x)= x$. It is clear that $\Phi$ is bijective, and the fact that $\pi\circ T=S\circ \pi$ implies that $\Phi\circ T = T\circ \Phi$. It remains to show that $\Phi$ is continuous.

Define $D:Y\rightarrow [0,\infty)$ by $D(y) = d(x,\tilde x)$. If $y_n\rightarrow y$ then  it can be verified that $D(y)\geq \limsup D(y_n)$, i.e. $D$ is upper semi-continuous. Suppose that $D$ is bounded away from 0. Then $\Phi$ is continuous. Let $x_n \to x$ and $\{x_n, \tilde{x}_n\} = \pi^{-1}(y_n)$. By compactness $\tilde{x}_n \to \tilde{x}$ along one subsequence and let us assume that $\tilde{x}_n \to x'$ along another subsequence, where $\tilde{x} \neq x'$. Since $D$ is bounded away from 0, we also have $x'\neq x$. But since $\pi$ is continuous, we have $\pi(x)=\pi(\tilde x)= \pi(x')= y$, a contradiction.

It remains to show that $D$ is bounded away from 0. Suppose that $D(y_n)\rightarrow 0$. The continuity of $T$ implies that for each $j$, $D(S^j y_n))\rightarrow 0$ as $n\rightarrow \infty$.  If $y_n\rightarrow y$, the minimality of $S$ implies that  in any open subset of $Y$ we can find points whose $D$ value is arbitrarily small. By upper semi-continuity, $\{ y: D(y)<r  \}$   is open and by the above it is also dense. Hence $D=0$ on a dense $G_\delta$  set, which is a contradiction.

\end{proof}

Unfortunately Theorem \ref{automorphism-almost-automorphic-v2} is no longer true for a uniform $k$-to-$1$ factor map, where $k > 2$.

\begin{example}\label{no-automorphism-example}
Let $\theta(a)=abb$, $\theta(b)=bac$, $\theta(c)=cca$, and let $\eta(x)=yxz$, $\eta(y)=yxx$, $\eta(z)=yxy$. 
Then $\eta$ has column number $1$, and $(X_\theta,\sigma)$ is an almost everywhere three-to-one extension of $(\Z_3,+1)$.
One verifies that the left radius one, right radius zero rule  in Table \ref{local-rule-table-1} 
defines an exactly three-to-one shift commuting map $\pi:X_\theta\rightarrow X_\eta$. However $(X_\theta,\sigma)$ has a trivial automorphism group.

\end{example}

\begin{table}[]\begin{tabular}{|l| |l|l|l|}
\hline
 & a  &  b & c  \\ \hhline{|=#=|=|=|}
 a&z  &y  &x  \\ \hline
 b&y  &x  &z  \\ \hline
 c&x  &z  &y  \\ \hline

\end{tabular}
\vspace{0.5cm}
 \caption{The local rule for the factor map in Example \ref{no-automorphism-example}.}
\label{local-rule-table-1}

\end{table}

\section{Strongly injective substitutions and language automorphisms}\label{strongly-injective-section}

\begin{definition}
Let $\mathcal L\subset\mathcal A^*$ be a language. We say that the bijection $\Phi:\mathcal L\rightarrow \mathcal L$ is an {\em $\mathcal L$-automorphism} if
$\Phi(w_1 w_2) = \Phi(w_1)\Phi(w_2)$ whenever $w_1$, $ w_2$ and $w_1w_2$ belong to  $\mathcal L $.
\end{definition}

We have the following proposition which links language automorphisms to automorphisms of the corresponding dynamical system. Recall that any shift space $X$ 
is defined by its language, which is closed under the taking of subwords, and where every word is left and right extendable to a word in $\mathcal L$. 
 Conversely, any language $\mathcal L$ satisfying these two properties 
defines a shift $X_\mathcal L$.

\begin{prop}
Let $\mathcal L$ be the language of a shift $(X,\sigma)$.  Then    $\Phi$ is an {\em $\mathcal L$-automorphism}  if and only if $\Phi\in \Aut(X,\sigma)$ has zero radius.\end{prop}

\begin{proof}

Suppose that $\Phi$  is an $\mathcal L$-automorphism. Since $\Phi$ is bijective, it must map letters to letters, i.e. $\Phi|_\mathcal A$ is a bijection $\phi:\mathcal A\rightarrow \mathcal A$. Since it is a language morphism, $\phi$ defines $\Phi$. We can extend $\Phi$, via concatenation, to a map on bi-infinite sequences $u$ in $X$, namely $\Phi((u_n)_{n\in \Z} = (\phi(u_n))_{n\in \Z}$.   Let $u\in X$ and let $v:=\Phi(u)$. Since $\Phi$ maps $\mathcal L$ to $\mathcal L$, $v\in X$ and so the image of  $\Phi(X)$ is contained in $X$. Since $\Phi$ is defined by a local rule it is  shift-commuting, and  it is injective on $X$ since it is injective on $\mathcal L$. By minimality, $\Phi$ is surjective.

The 
converse is straightforward.\end{proof}

\begin{remark}
In the substitutional case, we note that if $\Phi$ has radius zero then  $\kappa(\Phi)=0$.  For if $\Phi$ has radius zero, then for some $n$, $\Phi^n=\Id$. But then $n\kappa(\Phi)=0$, and  since $\Z_r$ is torsion free, our assertion follows.

\end{remark}
In an ideal world any automorphism $\Phi$ with $\kappa(\Phi)=0$ would have zero radius. Unfortunately this is not the case.

\begin{example} 
Let $\theta(a)=ab$, $\theta(b)=ca$ and $\theta(c)=ba$ and 
define the local rule $\phi$ with left radius one, as  $\phi(b)=a$, $\phi(c)=a$, $\phi(aa)=c$, $\phi(ba)=b$ and $\phi(ca)=b$.  Then $\phi$ is the local rule of an involution in $ \Aut(X_\theta,\sigma)$.
\end{example}

\begin{definition}\label{strongly}We say that $\theta$ is {\em strongly right- (left-) injective} if  $\theta$ is injective and does not have any  right- (left-) infinite fixed points which differ only in their 0-th entry, and we say that  $\theta$ is {\em strongly injective} if it is both strongly right- and  left-injective.
\end{definition}

\begin{definition}\label{minimal-set} Let  $\theta:\mathcal A\rightarrow \mathcal A^{r}$ be a length $r$ substitution with column number $c$. A {\em minimal set} (for $\theta$) is  a set of letters 
$\{\alpha_1 , \ldots , \alpha_c \}\subset \mathcal A$ of cardinality $c$ such that this set appears as $(\theta^n(\mathcal A))_i$ for some $n\in \N$ and some $0\leq i < r^n$. Let $\mathcal{X}$ denote the 
collection of all minimal sets:
\begin{align*}
	\mathcal{X} := \{M \in \mathcal{A}^{c(\theta)}: \exists k, j<r^k \text{ s.t. } \theta^k(\mathcal{A})_j = M\}.
\end{align*}
\end{definition}

\begin{thm}\label{strongly-injective-rep}
Let  $\theta$ be a  primitive length $r$ substitution 
such that $X_\theta$ is infinite. Then $(X_\theta, \sigma)$ is topologically conjugate to $(X_\eta, \sigma)$ where $\eta$ is  a strongly injective   length $r$ substitution.
\end{thm}

\begin{proof}
Note that if a substitution $\theta$ has column number $c$, then it has at least $c$ right-infinite and $c$ left-infinite fixed points. We claim that  if $\theta$ has exactly $c$ right- (left-) infinite fixed points, then it must be right- (left-) strongly injective. For, in this case there is a minimal set $\mathcal A_r$ ($\mathcal A_l$) of $c$ letters that generates the $c$ right- (left-) infinite fixed points. The fact that $\theta$ has column number $c$ implies that these
fixed points are pairwise distinct at each index $n\geq 0$, and  our claim follows. 

Now suppose that $\theta$ is not strongly injective. By replacing $\theta$ by a power if necessary, we can assume that there are three consecutive indices $k-1,k,k+1$, with  $k\in\{1, \ldots r-1\}$, where $\theta$'s column number is achieved jointly for these indices, i.e.    $|\{\theta(a)[k-1,k+1]: a\in \mathcal A\}| = c$.
  Now let $\bar \eta:=\theta^{(2,k)}$ be the $k$-shifted 2-sliding block representation of $\theta$. By construction, $\bar \eta$ has exactly $c$ right-infinite fixed points and $c$ left-infinite fixed points. If $\bar\eta$ is injective on letters, then by our previous claim, it is strongly injective and in this case 
 let $\eta=\bar\eta$.  If $\bar\eta$ is not injective on letters, 
define an equivalence relation on   its alphabet  where $\alpha \sim \beta$ if $\bar \eta (\alpha) =\bar  \eta(\beta)$. Let $\mathcal B$ be the set of equivalence classes for $\sim$ and define $\eta([\beta]):=\bar\eta(\beta) $. Then  $(X_\eta,\sigma)$ is topologically conjugate to $(X_{\bar \eta}, \sigma)$ via the map $\Phi((x_{n})_{n\in\Z}) = ([x_n])_{n\in\Z}$. (See for example  the proof of Theorem \ref{thm:injective} in \cite{BDM}.)   By construction $\eta$ is  injective on letters. Since  $\kappa (\phi)=0$,   the fixed points of $\bar\eta$ are in one-to-one correspondence with those of  $\eta$.  The result follows by our claim.

\end{proof}

\begin{thm}\label{radius0}
Let $\theta$ be a strongly injective primitive length $r$ substitution 
such that $X_\theta$ is infinite. If $\kappa(\Phi)=0$ 
then $\Phi$ 
has radius 0 and so in particular it is a language automorphism.

\end{thm}

\begin{proof}
We suppose that $\Phi$ is nontrivial; then it must permute fixed points, and it cannot fix either the right ray or the left ray of a fixed point. We assume that
 $\Phi$ sends the fixed point $u$, whose right-infinite ray is $u_{[0,\infty)}=\lim_k\theta^{k}(a)$,  to the  fixed point  $v\neq u$, whose right-infinite ray is  $v_{[0,\infty]}=\lim_k\theta^{k}(b)$.

 If $\Phi$ does not have radius 0, then there exist letters $c,d$ and $e$, with $d\neq e$, and indices $i\neq j$ such that 
$u_i=c=u_j$, $v_i=d$ and $v_j=e$. This implies that for all $n$, $\theta^n(c)$ appears starting at the indices  $r^ni$ and $r^nj$ in $u$,  $\theta^n(d)$ appears starting at index $r^ni$ in $v$, and $\theta^n(e)$ appears starting at index $r^nj$ in $v$.  Suppose that   the left radius of $\Phi$ is $\ell>0$. Then for all large $n$, $\theta^n(d)_{[\ell,r^n-1]}= \theta^n(e)_{[\ell,r^n -1]}$ 
and thus we have two  fixed points $u'=\ldots \cdot d'D'x$ and $v'=\ldots \cdot e'E'x$ which disagree  on their initial entries: $d'\neq e'$, where the length of $D'$ equals that of $E'$, and where the suffix of $D'$ does not equal that of $E'$, and finally where $u'$ and $v'$  agree on their right rays $x$ starting at index  at most $\ell$.
If $D'\neq \emptyset$, so that also $E'\neq \emptyset$, then the image of their suffixes under $\theta$ must agree, and this contradicts our assumption that $\theta$ is injective. (Here we are using the fact that $r$ must be at least two, of course.) Thus $D'=\emptyset=E'$
and thus we have two  fixed points $u'=\ldots \cdot d'x$ and $v'=\ldots \cdot e'x$ which disagree  on their initial entry: $d'\neq e'$, and which agree on their right rays starting at index 1. This contradicts our assumption that $\theta$ is strongly right-injective. If    $\Phi$ has right radius $r>0$, we would obtain a contradiction to our assumption that $\theta$ is strongly left-injective.
\end{proof}
%Let $\theta$ be a primitive length $r$ substitution on alphabet $\mathcal{A}$.
Our next goal is to characterize the automorphisms of $(X_{\theta}, \sigma)$ which have radius $0$. This is particularly interesting as, by Theorem~\ref{strongly-injective-rep} we can assume without loss of generality that $\theta$ is a strongly injective substitution.
Furthermore, by Theorem~\ref{radius0} we know that in this situation any $\Phi$ with $\kappa(\Phi) = 0$ has radius $0$. In particular the local rule is just a map from $\mathcal{A}$ to itself. 

The characterisation that we will give is inspired by the special case of bijective substitutions, and 
 it is instructive   to recall  this special case first. We note that bijective substitutions are strongly injective, so any element of $\ker \kappa$ is a letter exchanging map. Lemanczyk and Mentzen \cite[Theorem 5]{L-M} gave a closed form for the automorphism group of bijective substitutions in the measurable setting. In particular, they consider the group $G$ of permutations generated by the bijections $\{  \theta_i: 1\leq i \leq r \}$ of $\mathcal A$, and they show that a permutation $\tau$ defines a shift commuting
 measurable bijection  of $(X_\theta,\sigma)$ if and only if
it belongs to the centralizer of $G$ in the symmetric group $S_{|\mathcal A|}$.  This statement is still true if we pass to the topological setting. For, first the fact that $\ldots 00$ is the only singular fibre of the equicontinuous factor map
tells us that the  automorphisms of $X_\theta$ must have $\kappa$-value zero.  Our claim 
  can now  be seen using Proposition \ref{prop:last_nail_lemma}, which is effectively telling us that $\tau$ must commute with $G$. In fact we only need a simpler version of Proposition \ref{prop:last_nail_lemma}, namely one that takes into account the fact that $\tau$ must have radius zero.

We recall the definition of the column number in (\ref{column number}),
which  denotes the minimal number of elements of $\mathcal{A}$ one can see in one column when writing the words $\theta^k(a)$ for $a \in \mathcal{A}$ underneath each other.
Recall that we use  $\mathcal{X} \subseteq \mathcal{P}(\mathcal{A}^{c(\theta)})$ to denote the collection of all possible minimal sets for $\theta$.
The minimality of $M \in \mathcal{X}$ implies some useful properties: for any $a,b \in M$ we have that $\theta^k(a)_j \neq \theta^k(b)_j$ for all $k\geq 1$ and $j<r^k$. Otherwise, we would have a contradiction to the minimality of the column number $c$.
This implies that $\theta$ is well-defined on $\mathcal{X}$ and we denote it by $\tilde{\theta}$.
It follows quite easily that $\tilde{\theta}$ is primitive (see~\cite[Proposition 6.3]{L-M-2018}).

\begin{lem}\label{le:cover}
	$\mathcal{X}$ covers $\mathcal{A}$. In other words, every letter of $\mathcal A$ belongs to a set in $\mathcal{X}$.
\end{lem}
\begin{proof}
	This is Equation (37) in \cite{L-M-2018}, but we include the verification here for completeness.
	We need to show that for any $b \in \mathcal{A}$ there exists $M' \in \mathcal{X}$ with $b \in M'$.
	Obviously $\mathcal{X} \neq \emptyset$. Thus, there exists some $M$ and $a$ such that $a \in M \in \mathcal{X}$. As $\theta$ is primitive, there exist $k, j<r^k$ with $\theta^{k}(a)_j = b$. Now obviously $b \in \tilde{\theta}^k(M)_j =: M'$.
\end{proof}

This allows us to find our first restrictions for automorphisms with radius $0$.
\begin{lem}
	Suppose that $\tau$ is an automorphism of $(X_{\theta}, \sigma)$ with radius $0$.
	Then, for any $M \in \mathcal{X}$ we have that $\tau|_{M}$ is a bijection.
\end{lem}
\begin{proof}
	We know by a modified version of Proposition \ref{prop:last_nail_lemma}   (see for example \cite[Lemma 2]{B-R-Y}) that 
	\begin{align*}
		\tau \circ \theta^{c!} = \theta^{c!} \circ \tau.
	\end{align*}
	Thus by changing to $\theta^{c!}$ we can assume without loss of generality that
	\begin{align}\label{eq:tau_commute}
		\tau \circ \theta = \theta \circ \tau.
	\end{align}
	By the surjectivity of $\tau$ it follows directly that $\tau$ is a bijection on $\mathcal{A}$.
	We know that for any $M \in \mathcal{X}$ there exists $k \in \N, j< r^k$ with $\theta^{k}(\mathcal{A})_j = M$.
	Thus, for any $a \in M$ there exists $b \in \mathcal{A}$ such that $\theta^k(b)_j = a$.
	This gives by~\eqref{eq:tau_commute}
	\begin{align*}
		\tau(a) = \tau \circ \theta^k_j(b) = \theta^k_j \circ \tau(b) \in M.
	\end{align*}
	So $\tau$ maps $M$ to itself and as $\tau$ is injective it is bijective on $M$.
\end{proof}

So we know that $\tau$ is a bijection on every minimal set and by~\eqref{eq:tau_commute} we would expect these bijections to be compatible. Next we aim to relate elements of minimal sets to each other.
Therefore we fix some $M_0 \in \mathcal{X}$ and $k_0$  with
\begin{align}\label{definition-of-M_0}
 	\theta^{k_0}(\mathcal{A})_{j_0} = M_0 \mbox{ for some } j_0 < r^{k_0}.
\end{align}
Furthermore, we can assume without loss of generality that $\theta^{k_0}(a)_{j_0} = a$ for every $a \in M_0$.
This is due to the fact that $\theta^{k_0}_{j_0}$ defines a bijection on $M_0$ and taking a power yields the result.
Now we fix an arbitrary bijection $f_0: M_0 \to \{1,\ldots,c\}$ which allows us to define $f: \mathcal{A} \to \{1,\ldots,c\}$ by
\begin{align}\label{definition-of-f}
	f(a) = f_0(\theta^{k_0}(a)_{j_0}).
\end{align}
In particular, we have $f|_{M_0} = f_0$. Note also that the minimality of $c$ implies that  for any $M \in \mathcal{X}$ we have that 
$f|_{M}$
 is a bijection between $M$ and $\{1,\ldots,c\}$.
\begin{lem}
	There exists a bijection $\tau'$ on $\{1,\ldots,c\}$ such that $f \circ \tau = \tau' \circ f$.
\end{lem}
\begin{proof}
	Let us take $a_1 \neq a_2$ such that $f(a_1) = f(a_2)$. We need to show that $f(\tau(a_1)) = f(\tau(a_2))$.
	By the definition of $f$ we have 
	\begin{align*}
		\theta^{k_0}(a_1)_{j_0} = \theta^{k_0}(a_2)_{j_0}.
	\end{align*}
	By~\eqref{eq:tau_commute} this gives
	\begin{align*}
		f(\tau(a_1)) &= f_0(\theta^{k_0}_{j_0} \circ \tau (a_1)) = f_0(\tau \circ \theta^{k_0}_{j_0}(a_1))\\
			& = f_0(\tau \circ \theta^{k_0}_{j_0}(a_2)) = f_0(\theta^{k_0}_{j_0} \circ \tau (a_2)) = f(\tau(a_2)).
	\end{align*}
\end{proof}

However, there is still one more condition that $\tau$ has to satisfy. For this remaining condition we use a generalisation of Lemanczyk and Mentzen's idea for bijective substitutions to assign to a general $\theta$ a group of permutations called $G$. The following was presented by the first author and Lemanczyk in~\cite{L-M-2018} and we follow their notation. Let us fix an arbitrary $M \in \mathcal{X}$. Then $\theta_j$ is a bijection from $M$ to $\tilde{\theta}(M)_j$. Thus it corresponds via $f$ to a permutation of $\{1,\ldots,c\}$ which we denote by $\sigma_{M,j}$. More precisely we have $\sigma_{M,j}(m) = n$ if there exist $a \in M, b \in \tilde{\theta}(M)_j$ with $f(a) = n, f(b) = m$ such that $\theta(a)_j = b$.
Another way to define it is via the relation
\begin{align}\label{eq:sigma_Mj}
	\sigma_{M,j}^{-1}(f(a)) = f(\theta(a)_j).
\end{align}
Furthermore, we denote by $G$ the group generated by all the $\sigma_{M,j}$:
\begin{align*}
	G(\theta) := <\sigma_{M,j}: M \in \mathcal{X}, 0\leq j \leq r-1>.
\end{align*}
We have gathered now all the information needed for the following Theorem. In what follows the function $f$ is that defined in Equation \eqref{definition-of-f}.
\begin{thm}\label{characterisation-of-kernel}
	Let $\theta$ be a primitive and injective length $r$  substitution of height $1$ 
such that $X_\theta$ is infinite.
	Then $\tau$ is an automorphism with radius $0$ if and only if
	\begin{enumerate}[(i)]
		\item $\tau$ is a bijection on every minimal set $M \in \mathcal{X}(\theta)$.
		\item $\tau$ is well-defined on $\{1,\ldots,c\}$ via $f$, which we denote by $\tau'$, i.e. $f \circ \tau = \tau' \circ f$.
		\item $\tau'$ belongs to the centralizer of $G(\theta)$.
	\end{enumerate}
\end{thm}
\begin{proof}
	First we still need to show that any automorphism $\tau$ with radius $0$ fulfills (iii):
	Obviously it is sufficient to show that $\tau'$ commutes with every $\sigma^{-1}_{M,j}$.
	We find for any $a \in M$,
	\begin{align*}
		\tau' \circ \sigma^{-1}_{M,j} (f(a)) &= \tau' \circ f(\theta(a)_j)) = f \circ \tau (\theta(a)_j) = f(\theta(\tau(a))_j)\\
			& = \sigma^{-1}_{M,j} \circ f \circ \tau(a) = \sigma^{-1}_{M,j} \circ \tau' (f(a)).
	\end{align*}
	Here the first and fourth equation are due to~\eqref{eq:sigma_Mj}, the second and fifth equation follow from (ii) and the third equation from~\eqref{eq:tau_commute}.
	
	It remains to show that any $\tau$ that satisfies (i) - (iii) indeed gives an automorphism. It is sufficient to check~\eqref{eq:tau_commute}  as this implies $\tau(\theta^k(a)[j-1,j+1]) = \theta^k(\tau(a))[j-1,j+1] \in \mathcal{L}_{\theta}$.
	Let $a \in \mathcal{A}$ be an arbitrary element and $j<r$. 
	By Lemma~\ref{le:cover} we know that $a \in M'$ for some $M' \in \mathcal{X}$.
	By (i) we have that $\tau(a) \in M'$. It follows by the definition of $\tilde{\theta}$ that $\theta(a)_j, \theta(\tau(a))_j \in \tilde{\theta}(M')_j =: M$. 
	This gives
	\begin{align*}
		\theta(\tau(a))_j &= f|_{M}^{-1} \circ f (\theta(\tau(a))_j)\\
			&= f|_{M}^{-1} \circ \sigma_{M,j}^{-1} \circ f \circ \tau(a)\\
			&= f|_{M}^{-1} \circ \sigma_{M,j}^{-1} \circ \tau' \circ f (a)\\
			&= f|_{M}^{-1} \circ \tau' \circ \sigma_{M,j}^{-1} \circ f(a)\\
			&= f|_{M}^{-1} \circ \tau' \circ f(\theta(a)_j)\\
			&= \tau \circ f|_{M}^{-1} \circ f(\theta(a)_j)\\
			&= \tau (\theta(a))_j.
	\end{align*}
	Here the second and fifth equation hold by~\eqref{eq:sigma_Mj}, the third and sixth equation by (ii) and the fourth equation by (iii).
\end{proof}

\begin{remark}
	In the case of bijective substitution we have that for any element in the centralizer of $G$ there exists an automorphism. This is no longer the case in the general situation as the following example shows.
\end{remark}
\begin{example}\label{centraliser-larger-automorphism}
	We consider the strongly injective substitution
	\begin{align*}
		\theta(a) &= aac\\
		\theta(b) &= bba\\
		\theta(c) &= bca.
	\end{align*}
	We find directly that $\mathcal{X} = \{\{a,b\}, \{a,c\}\}$.
	We have $\theta(a)_0 = a, \theta(b)_0 = b, \theta(c)_0 = b$. Thus we define $f(a) = 1, f(b) = f(c) = 2$.
	This gives
	\begin{align*}
		\sigma_{\{a,b\}, 0} &= id\\
		\sigma_{\{a,b\}, 1} &= id\\
		\sigma_{\{a,b\}, 2} &= (12)\\
		\sigma_{\{a,c\}, 0} &= id\\
		\sigma_{\{a,c\}, 1} &= id\\
		\sigma_{\{a,c\}, 2} &= (12).
	\end{align*}
	Therefore, we have $G(\theta) = \{id, (12)\}$ and the centralizer of $G$ is $G$ itself.
	However, as $a \in \{a,b\}$ and $a \in \{a,c\}$ we know by (i) that $\tau(a) = a$ and therefore $\tau(b) = b, \tau(c) = c$.
	This gives an automorphism with $\tau = id$ which corresponds to $\tau' = id$, but we don't have any automorphism corresponding to $\tau' = (12)$.
\end{example}

Because of  the situation in Example \ref{centraliser-larger-automorphism}, Theorem \ref{characterisation-of-kernel} is not entirely satisfactory. We recall that Lemanczyk and Mentzen were working in a measurable context, and this raises the question of whether we can hope for a better version of Theorem \ref{characterisation-of-kernel} if we move to the measurable setting. It turns out that this is indeed the case, as we shall next show in Theorem \ref{best-result-yet}.  To do this we recall the notion of a {\em reduced} substitution, due to Host and Parreau \cite{HP}.

We denote by $d^{*}: (\mathcal{A}^2)^{*} \to \R$,
\begin{align*}
	d^{*}(a_1\ldots a_n, b_1\ldots b_n) = \frac{|\{1\leq i \leq n: a_i \neq b_i\}|}{n}.
\end{align*}
\begin{definition}
	We call a primitive length $r$ substitution \emph{reduced} if for all $a,b\in \mathcal{A}$ we have
	\begin{align*}
		\lim_{k\to\infty} d^{*}(\theta^{k}(a), \theta^{k}(b)) > 0.
	\end{align*}
\end{definition}
One can reduce a substitution by identifying all $a,b\in \mathcal{A}$ for which
\begin{align*}
	\lim_{k\to\infty} d^{*}(\theta^{k}(a), \theta^{k}(b)) = 0.
\end{align*}
When reducing a substitution, the related dynamical system does not change measure-theoretically.

\begin{lem}\label{le_reduce_f}
	Let $\theta$ be a primitive length $r$ substitution.
	Then we have for any $a,b \in \mathcal{A}$
	\begin{align}\label{eq_reduce_ab}
		\lim_{k\to \infty} d^{*}(\theta^{k}(a), \theta^{k}(b)) = 0
	\end{align}
	if and only if  for each $k$, we have
	\begin{align}\label{eq_f_equal}
		f(\theta^{k}(a)_j) = f(\theta^{k}(b)_j) \text{ for all }  j<r^{k}.
	\end{align}
\end{lem}
\begin{proof}
	Let $a\in M\in \mathcal{X}(\theta)$. Then $\theta^{k}(a)_j$ is uniquely determined by $f(\theta^{k}(a)_j)$ and $\tilde{\theta}^k(M)_j$.
	Furthermore, as $c(\tilde{\theta}) = 1$ we know that
	\begin{align*}
		\lim_{k\to\infty}d^{*}(\tilde{\theta}^{k}(M_1), \tilde{\theta}^{k}(M_2)) = 0,
	\end{align*}
	for all $M_1,M_2 \in \mathcal{X}(\theta)$. This shows that \eqref{eq_f_equal} implies \eqref{eq_reduce_ab}.\\
	Let us now assume that there exists $k'$ and  $j'<r^{k'}$ such that
	\begin{align*}
		f(\theta^{k'}(a)_{j'}) \neq f(\theta^{k'}(b)_{j'}).
	\end{align*}
	Recall  the definition of $M_0$, $k_0$, $j_0$ and $f$ via $(\ref{definition-of-M_0})$ and  $(\ref{definition-of-f})$. We know that
	\begin{align*}
		f_0(\theta^{k_0+k'}(a)_{j' r^{k_0} + j_0}) =  f_0(\theta^{k_0}(\theta^{k'}(a)_{j'})_{j_0})\neq f_0(\theta^{k_0}(\theta^{k'}(b)_{j'})_{j_0}) = f_0(\theta^{k_0+k'}(b)_{j' r^{k_0} + j_0}).
	\end{align*}
	As $f_0$ is a bijection between $\{1,\ldots,c\}$ and  $M_0$ we know that
	\begin{align*}
		a' := \theta^{k_0+k'}(a)_{j' r^{k_0} + j_0}\neq \theta^{k_0+k'}(b)_{j' r^{k_0} + j_0} =: b'
	\end{align*}
	and also $a', b' \in M_0$.
	We know by the definition of minimal sets that for all $k $ and $ j<r^{k}$
	\begin{align*}
		\theta^{k}(a')_{j} \neq \theta^{k}(b')_{j}.
	\end{align*}
	Thus we find that
	\begin{align*}
		\lim_{k\to \infty} d^{*}(\theta^{k}(a), \theta^{k}(b)) \geq \frac{1}{r^{k'+k_0}}.
	\end{align*}
\end{proof}

\begin{cor}\label{automorphism=centraliser}
	Let $\theta$ be a reduced and primitive substitution of constant length $r$.
	Then every $\tau'$ in the centraliser of $G(\theta)$ defines an automorphism $\tau\in \Aut(X_\theta,\sigma).$
\end{cor}
\begin{proof}
	As $\theta$ is reduced this implies that $\theta$ is strongly injective, so we know that any automorphism with $\kappa$-value $0$ has radius $0$. 
	We need to define $\tau: \mathcal{A} \to \mathcal{A}$ that satisfies the conditions of Theorem~\ref{characterisation-of-kernel}. Fix any $f$ defined as in (\ref{definition-of-f}).
	Let $a \in M \in \mathcal{X}$. Then we define
	\begin{align}\label{eq_def_tau}
		\tau(a) = f|_{M}^{-1}(\tau'(f(a))),
	\end{align}
	and we need to make sure that this definition is independent of the choice of $M$:\\
	Let $a \in M_1 \cap M_2$. 
	Then we set
	\begin{align*}
		b_i = f|_{M_i}^{-1}(\tau'(f(a))).
	\end{align*}
	So we have $b_i \in M_i$ and
	\begin{align*}
		f(b_i) = \tau'(f(a)).
	\end{align*}
	We compute
	\begin{align*}
		\tau'(f(\theta^{k}(a)_j)) &= \tau' \circ (\sigma_{M_i,j}^{(k)})^{-1} (f(a))\\
			&= (\sigma_{M_i,j}^{(k)})^{-1}(\tau'(f(a)))\\
			&= (\sigma_{M_i,j}^{(k)})^{-1}(f(b_i))\\
			&= f(\theta^{k}(b_i)_j).
	\end{align*}
	for $i = 1,2$.
	Thus we know that $f(\theta^{k}(b_1)_j) = f(\theta^{k}(b_2)_j)$ for all $k, j<r^k$. By Lemma~\ref{le_reduce_f} we have
	\begin{align*}
		\lim_{k\to\infty} d^{*}(\theta^k(b_1), \theta^k(b_2)) = 0.
	\end{align*}
	As $\theta$ is reduced we conclude that $b_1 = b_2$, so that~\eqref{eq_def_tau} is well defined and one checks easily that $\tau$ fulfils the conditions of Theorem~\ref{characterisation-of-kernel}.
\end{proof}

\begin{lem}\label{centralisers-are-equal}
	Let $\theta$ be a primitive length $r$ substitution and let $\theta'$ be the reduced substitution of $\theta$.
	Then  $G(\theta) = G(\theta')$.
\end{lem}
\begin{proof}
	 Let $\mathcal A'$ be the alphabet on which $\theta'$ is defined. By Lemma~\ref{le_reduce_f}  we can define $f$ on $\mathcal{A}'$ as in particular it says that   if $a \sim b$ then $f(a) = f(b)$, so we can simply define $f([a]) := f(a)$. This shows that $|\theta'^k([M])_j| = c(\theta)$ for all $k$ and  $j<r^k$ which means that $c(\theta') \geq c(\theta)$. As $\theta'$ is constructed by identifying letters of $\mathcal{A}$ we have $c(\theta) \geq c(\theta')$ and so $c(\theta) = c(\theta')$.
	It suffices to show that
	\begin{align}\label{eq_G_theta'}
		\{\sigma_{M,j}: M \in \mathcal{X}(\theta), j<r\} = \{\sigma'_{M,j}: M \in \mathcal{X}(\theta'), j<r\}.
	\end{align}
	First we find directly that for all $M \in \mathcal{X}(\theta)$ we have $[M] \in \mathcal{X}(\theta')$ and as $f([a]) = f(a)$ we know that
	\begin{align*}
		\{\sigma_{M,j}: M \in \mathcal{X}(\theta), j<r\} \subseteq \{\sigma'_{M,j}: M \in \mathcal{X}(\theta'), j<r\}.
	\end{align*}
	Furthermore, for any $M' \in \mathcal{X}(\theta')$ there exist $k'$ and  $j' < r^{k'}$ such that $\theta'^{k'}(\mathcal{A}')_{j'} = M'$.
	We recall that as in (\ref{definition-of-M_0}), we have $M_0 = \theta^{k_0}(\mathcal{A})_{j_0}$ and, therefore, we find
	\begin{align*}
		\mathcal{X}(\theta) \ni \theta^{k_0 + k'}(\mathcal{A})_{j_0 r^{k'} + j'} = \theta^{k'}(M_0)_{j'} \subseteq M'
	\end{align*}
	Thus, we can write $M' = [M]$ for some $M \in \mathcal{X}(\theta)$, which shows that
	\begin{align*}
		\{\sigma_{M,j}: M \in \mathcal{X}(\theta), j<r\} \supseteq \{\sigma'_{M,j}: M \in \mathcal{X}(\theta'), j<r\}
	\end{align*}
	and in total~\eqref{eq_G_theta'}, which finishes the proof.
\end{proof}
Recall that if $\theta$ is primitive, then there is a unique probability measure $\mu$ such that $(X_\theta,\sigma,\mu)$ is a measure preserving dynamical system. The {\em essential centraliser} of 
$(X_\theta,\sigma,\mu)$ is the set of all measure-preserving maps of $X_\theta$ which commute with $\sigma$, quotiented out by $\{ \sigma^n: n\in \Z\}$. Combining Corollary \ref{automorphism=centraliser} and Lemma \ref{centralisers-are-equal}, we deduce the following.
\begin{thm}\label{best-result-yet}
	Let $\theta$ be a primitive length $r$ substitution. Then $G(\theta)$ corresponds to  those elements of the essential centraliser of  $(X_\theta,\sigma,\mu)$  that are a root of the identity.

\end{thm}

\section{Rational orders in an automorphism group, twisted substitutions and  $k$-compressions. }
Theorem~\ref{strongly-injective-rep} tells us that we can assume that our substitution is strongly injective, and henceforth we do this.  Theorem~\ref{radius0} tells us that all automorphisms whose $\kappa$-value is zero arise from alphabet permutations. 
In this section we investigate the structure of substitutions $\theta$ which possess other automorphisms, namely, automorphisms with a non-integer $\kappa$-value. We show in Theorem \ref{roots} that these automorphisms indicate that $\theta$ has a structure coming from some other constant length substitution.

Given a sequence $(x_n)_{n\in \Z}\in \mathcal A^{\Z}$, we can consider what Cobham \cite{Cobham} calls its {\em $k$-compression},  which is the sequence  $(y_n)_{n\in \Z}\in ({\mathcal A^{k}})^{\Z}$ where $y_n := x_{nk} \ldots x_{(n+1)k-1}$. 
The operation of $k$-compression also applies to words whose length is a multiple of $k$. If $\theta:\mathcal A\rightarrow\mathcal A^{r}$, and $\alpha_1\ldots \alpha_k\in \mathcal A^{k}$, then the $k$-compression of the word $\theta( \alpha_1\ldots \alpha_k  )$ is a word in $\mathcal A^{k}$ of length $r$; this defines a substitution $\theta^{(k)}: \mathcal B\rightarrow \mathcal B^{r}$, where $\mathcal B\subset   \mathcal A^{k}$
is the set of words      $\mathcal B = \{u[mk, (m+1)k-1]: m \geq 0\}$ where    $(u_n)_{n\in \N}$ is a fixed point of $\theta$.
Cobham  studied the notion of $k$-compression and showed that if $x$ is $r$-automatic and primitive, then for any $k$, its $k$-compression is also $r$-automatic and primitive, generated by the substitution $\theta^{(k)}$. We call $\theta^{(k)}$ (resp. $(X_{\theta^{(k)}}, \sigma)$) the {\em $k$-compression of $\theta$ }  (resp. $(X_\theta,\sigma) $). It is straightforward to see that $(X_\theta,\sigma^{k})$ is topologically conjugate to $(X_{\theta^{(k)}},\sigma)$.

\begin{definition}\label{twist-def}
Let $\theta:\mathcal A \rightarrow \mathcal A^{r}$ be a substitution and let $\tau \in \Aut(X_\theta, \sigma)$ with radius $0$ satisfying $\kappa(\tau)=0$ and   $\tau^{r}=\tau$.
Define the substitution  $\theta_\tau:\mathcal A\rightarrow \mathcal A^{r}$ by 
\[\theta_\tau(\alpha)_j =   \tau^j(\theta(\alpha)_j).\]
We call $(X_{\theta_\tau},\sigma)$ the  {\em $\tau$-twist of $(X_\theta,\sigma)$}. 

\end{definition}

Recall $\kappa_\theta: \Aut(X_\theta,\sigma)\rightarrow \Z_r$,  the group homomorphism defined in Theorem \ref{thm:ethan}. By \cite[Theorem 3.16 and Lemma 3.2]{coven-quas-yassawi},  $\kappa (\Aut(X_{\theta}, \sigma)) $ is always cyclic and generated by $1/k$ for some $k$.

\begin{thm}\label{roots}
Let  $\theta$ be a strongly injective primitive length $r$ substitution  of height one
such that $X_\theta$ is infinite. 
If 
$\kappa_\theta (\Aut(X_{\theta}, \sigma)) = \langle \frac{1}{k} \rangle$ then   $(X_\theta,\sigma)$ is topologically conjugate to  a twist of a  $k$-compression of  $(X_\eta,\sigma)$, with $\eta$ a primitive length $r$ substitution.

\end{thm}

It is often the case that  $\tau=\Id$, and in this case one can conclude that $(X_\theta, \sigma)$ is topologically conjugate to a $k$-compression.
However, situations where $\tau$ is nontrivial do occur; see 
Example \ref{twisted}.

To prove Theorem \ref{roots} we prove some  preliminary lemmas.

\begin{lem}\label{roots-one}
Let  $\theta$ be a primitive length $r$ substitution  of height one
such that $X_\theta$ is infinite. 
If 
$ \frac{1}{k}\in \kappa(\Aut(X_\theta,\sigma))$ then, by taking a power of $\theta$ if necessary,  there is an automorphism $\Phi$ with $\kappa(\Phi) = \frac{1}{k}$ satisfying
\begin{equation}\label{eq:thetaPhicomm}
\Phi^r\circ\theta=\theta\circ\Phi ,
\end{equation}
and an 
 automorphism $\tau\in \Aut(X_\theta,\sigma)$ such that $\tau^r=\tau$ and 
 \begin{equation}\label{eq:Phi-tau-link}
 \Phi^{k}= \sigma \circ \tau.
\end{equation}
\end{lem}

\begin{proof}
Let $\Phi $ be any automorphism with $\kappa(\Phi)=\frac{1}{k}$. Then we have
 \[\pi(\Phi^r\circ \theta(x))= \frac{r}{k}+\pi(\theta(x))=r\left(\frac{1}{k}+\pi(x)\right),\] and since  $\frac{1}{k} \in \Z_r$, this implies that $\Phi^{-1}\circ\theta^{-1}\circ\Phi^r\circ\theta:X_\theta\rightarrow X_\theta$ is a well defined bijection which also commutes with $\sigma$, i.e. is an automorphism. We  also have
 $\kappa(\Phi^{-1}\circ\theta^{-1}\circ\Phi^r\circ\theta)= 0$.      We claim, by taking a power of $\theta$ if necessary,  that there exists an automorphism $\Phi$ such that $\kappa(\Phi)=\frac{1}{k}$ and Equation \eqref{eq:thetaPhicomm} holds.

If   $\kappa$ is injective, then $\kappa(\Phi^{-1}\circ\theta^{-1}\circ\Phi^r\circ\theta)= 0$ implies that $\Phi^{-1}\circ\theta^{-1}\circ\Phi^r\circ\theta=\Id$ and we are done.
If $\kappa$ is not injective, consider the map $f: \kappa^{-1}(\frac{1}{k})\rightarrow \kappa^{-1}(\frac{1}{k})$ given by    $\Phi \mapsto \theta^{-1}\Phi^r\theta$. For some $n\geq 1$, and for some $\Phi$, we have $f^{n}(\Phi)=\Phi$. One verifies that in this case $\Phi$ satisfies $\theta^n\Phi= \Phi^{r^n} \theta^n$. As $X_{\theta^n}=X_\theta$, we can replace $\theta$ by $\theta^n$ if necessary, to assume that we started with a $\theta$ for which there exists an automorphism $\Phi$ satisfying \eqref{eq:thetaPhicomm}. 

For this $\Phi$ satisfying $ \eqref{eq:thetaPhicomm}   $, we have that $\Phi^{k}= \sigma \circ \tau$ for some $\tau\in \ker{\kappa}$. It remains to show that $\tau^r=\tau$.
Considering the power $\theta^{c!p}$ instead of $\theta$, we have by Equation \eqref{eq:commuting-relation} that 
$ \sigma^N\circ 
\Phi\circ \theta =\theta\circ \Phi$ for some integer $N$. Using Equation \eqref{eq:thetaPhicomm}, we have $ \sigma^N\circ 
\Phi\circ \theta =\Phi^r\circ \theta$, so that $\sigma^N\Phi= \Phi^{r}$. Using Equation \eqref{eq:Phi-tau-link}, we have
\[ \sigma^r\tau^r = \Phi^{kr} =\sigma^{Nk}\Phi^k = \sigma^{Nk+1}\tau . \]
Both sides being automorphisms with necessarily equal $\kappa$-values, this implies that $Nk+1=r$, and the result follows.

\end{proof}

\begin{lem}\label{roots-two}

Let  $\theta$ be a strongly injective primitive length $r$ substitution  of height one
such that $X_\theta$ is infinite. 
If  $\Phi\in (X_\theta, \sigma)$  satisfies \begin{equation*}
\kappa(\Phi) = \frac{1}{k},\,\,\,\,
\Phi^r\circ\theta=\theta\circ\Phi  \mbox{ and }  \Phi^{k}= \sigma \circ \tau
\end{equation*}
 for some automorphism $\tau\in \Aut(X_\theta,\sigma)$ such that $\tau^r=\tau$,
then there exists a substitution $\eta$ such that $(X_\theta,\Phi)$ is topologically conjugate to  $(X_\eta, \sigma)$.

\end{lem}
\begin{proof}
By Theorem \ref{radius0}, $\tau$ has radius zero. If $\theta$ is defined on the alphabet $\mathcal A$, define a map $\pi \colon X_{\theta} \to \mathcal A^\Z$ by
$\pi(x)_n=(\Phi^{n} x)_0$. As we assume that $k>1$,  
then by an analogue of Proposition \ref{prop:radius_two}, $\Phi$ has left radius 0 and right radius at most 1. We notice that 
\begin{equation}\label{eq:piPhicomm}
\pi\circ\Phi=\sigma\circ\pi,\end{equation}
where we use $\sigma$ to denote the shift on either of our spaces. Let $\bar X:=\pi(X_{\theta})$;   then $\bar X$ is closed and shift invariant. We have
\[           \pi(x)_{nk}= \left(     \Phi^{nk}(x)   \right)_{0}  = (\sigma^n\tau^n(x))_0 = (\tau^n(x))_n = \tau^n(x_n),  \]
and since $\tau$ must be injective on letters, this implies that $\pi$ is injective.

Since  $\pi$ conjugates $(X_\theta,\Phi)$ to $(\bar X,\sigma)$, and $\Phi^k = \sigma\circ \tau$, then $\pi$ conjugates $(X_\theta,\sigma\circ \tau)$ to
$(\bar X,\sigma^k)$. Since $(X_\theta, \sigma)$ is minimal and $\tau$ is an automorphism, then $(X_\theta,\sigma\circ \tau)$ is minimal and so $(\bar X,\sigma^k)$ is minimal; hence 
$(\bar X,\sigma)$ is also minimal. 

We define a map on $\bar X$ by $\eta=\pi\circ   \theta   \circ\pi^{-1}$.
We have
\begin{equation}\label{eq:thetacomm}
\eta\circ \sigma=\pi\circ \theta\circ \pi^{-1}\circ \sigma
=\pi\circ\theta\circ\Phi \circ\pi^{-1}
=\pi\circ\Phi^{ r}\circ\theta\circ\pi^{-1}=\sigma^r\circ\eta,
\end{equation}
where the second, third and fourth equalities are by
\eqref{eq:piPhicomm}, \eqref{eq:thetaPhicomm}   
 and \eqref{eq:piPhicomm} 
respectively.

We claim that $\eta$ is a length $r$ substitution. 
By \eqref{eq:thetacomm}, it suffices to check that if $z,z'\in \bar X$
satisfy $z_0=z_0'$, then $\eta(z)_i=\eta(z')_i$
for $i=0,\ldots,r-1$. Suppose then that $z$ and $z'\in\bar X$ satisfy
$z_0=z'_0$. Let $x=\pi^{-1}z$ and $x'=\pi^{-1}z'$.
Notice that if $z=\pi(x)$, then $z_0= x_0$, so that
$x_0=x'_0$. Hence $\theta(x)_i=\theta (x')_i$
for $0\le i<r$. As $\Phi$ has right radius at most one and left radius 0, then 
$\eta(z)_i=(\pi\circ\theta (x))_i=(\pi\circ\theta (x'))_i=\eta(z')_i$
for $0\le i\le r-1$.
It follows that $\eta$ is a length $r$ substitution map, as required.

If $x\in X_{\theta}$ is a fixed point of $\theta$, then $\pi(x)\in \bar X$
is a fixed point of $\eta$. Furthermore, by minimality its orbit is
dense in $\bar X$, so that $\bar X$ is the substitution space of $\eta$.
By \eqref{eq:piPhicomm}, $\pi$ conjugates $\Phi \colon X_{\theta}\to X_{\theta}$
to $\sigma:X_\eta \to X_\eta$.  The result follows.

 \end{proof}

   \begin{proof}[Proof of Theorem \ref{roots}]

   We continue with the notations of Lemmas \ref{roots-one} and $\ref{roots-two}$.
   By Lemma \ref{roots-one}, there exist automorphisms $\Phi$ and $\tau$, with $\kappa$ values $\frac{1}{k}$ and 0 respectively, and  satisfying Equations \eqref{eq:thetaPhicomm} and \eqref{eq:Phi-tau-link}.     Also $\tau^r=\tau$.   By Lemma \ref{roots-two}, $(X_\theta, \Phi^k)= (X_\theta,\sigma\circ \tau)$ is conjugate to the $k$-compression  $(X_{\eta^{(k)}},\sigma)$  of the substitution shift $(X_\eta,\sigma)$.
   Let $\bar \pi:  (X_\theta,\sigma\circ \tau)\rightarrow   (X_{   \eta^{(k)} },\sigma)$ be a conjugacy.
   Since $\tau^{-1}$ is an automorphism
 of     $(X_\theta,\sigma\circ \tau)$, this implies that  $(X_\theta,\sigma)$ is conjugate to  $  (X_{ \eta^{(k)}},\sigma\circ \bar \tau) $
 for some automorphism $\bar \tau$ of $(X_{\eta^{(k)}},\sigma)$ with $\bar\tau\in \ker\kappa_{\eta^{(k)}}$.

   We claim that $\bar\tau$ has radius zero. To see this it is sufficient to show that $\eta$ is strongly injective, as this will imply that $\eta^{(k)}$  
 is also strongly injective and we are done by Theorem \ref{radius0}.

  Note that the column number of $\eta$ is at most the column number of $\theta$. If $\theta$ has column number $c$, we can assume, by considering  the rotation in the proof of Theorem 
 \ref{strongly-injective-rep} if necessary, where we replace $\theta$ by  the appropriate $k$-shifted 2-sliding block representation,
  that $\theta$ has exactly $c$ right-infinite fixed points and $c$ left-infinite fixed points. The map $\Phi$ has left radius zero and right radius one, and  each right-infinite $\theta$-fixed point
 is mapped to a right-infinite $\eta$-fixed point. Furthermore any right-infinite $\eta$-fixed point  is the image under $\pi$, of a right-infinite  $\theta$-fixed point. 
 Thus $\eta$ has $c$ right-fixed points and hence its column number must also be $c$. Now since $\kappa(\Phi^{-1})=-\kappa(\Phi^{-1})$,  $\Phi^{-1}$ has left radius one and right radius 0. Working with $\Phi^{-1}$
 (i.e. moving up this spacetime diagram, instead of down it), we see as above that $\eta$ has exactly $c$ left fixed points.
  By the discussion in
 the proof of Theorem \ref{strongly-injective-rep}, $\eta$ is strongly injective.

 Thus $\bar\tau$ has radius zero. Suppose that $\eta^{(k)}$ is defined on an alphabet $\mathcal B$. Define a map $p:X_{\eta^{(k)}} \rightarrow \mathcal B^{\Z}$ by $p(x)_n=((\sigma\circ \bar \tau)^{n}( x))_0$. One verifies that if $x\in X_{\eta^{(k)}}$ is a fixed point for $\eta^{(k)}$, then $p(x)_n = {\bar \tau}^{n}(x_n)$, i.e. 
 $p(x)$ is a 
  $\left({\eta^{(k)}}\right)_{\bar \tau}$-fixed point. Now $p $ is an injection, since $\bar \tau$ is an automorphism with  radius zero.
 Furthermore $p\circ\sigma\circ\bar\tau = \sigma\circ p$, i.e. $p$ is a conjugacy between $(X_{\eta^{(k)}},\sigma\circ\bar \tau)$ and  $(X_{ ({\eta^{(k)}} )_{\bar\tau}  }   ,    \sigma)$.   We have shown that
$(X_\theta,\sigma)$ is conjugate to  $  (X_{ \eta^{(k)}},\sigma\circ \bar \tau) $, which is conjugate to $(X_{ ({\eta^{(k)}} )_{\bar\tau}  }   ,    \sigma)$, i.e. that $(X_\theta,\sigma)$ is conjugate to the twist of a $k$-compression.
 \end{proof}

We specialise to the following case,  which includes the case  when $\Aut(X_\theta, \sigma)$ is cyclic, generated by an automorphism whose $\kappa$-value is $\frac{1}{k}$.  This happens, for instance, if $\theta$
has a coincidence, and more generally
 if $\kappa$ is injective and $\theta$ has height
coprime to $k$  \cite[Corollary 3.7]{coven-quas-yassawi}. Then we can restate our result as follows.

\begin{cor}\label{roots-corollary}
Let $\theta$ be a primitive length $r$ substitution  
such that $X_\theta$ is infinite. If $\Phi\in \Aut(X_\theta,\sigma)$ and $\Phi^k=\sigma$
then $(X_\theta,\sigma)$ is topologically conjugate to the $k$-compression of a substitution shift $(X_{\eta},\sigma)$.

\end{cor}

The following example shows that the twist in Theorem~\ref{roots} is indeed necessary.
\begin{example} \label{twisted}
Consider the substitution
\begin{align*}
	\theta(a) &= afbcb\\
	\theta(b) &= ahfgf\\
	\theta(c) &= afbdf\\
	\theta(d) &= ahfhb\\
	\theta(e) &= ebfgf\\
	\theta(f) &= edbcb\\
	\theta(g) &= ebfhb\\
	\theta(h) &= edbdf.
\end{align*}
It has an automorphism  $\Phi$ with $\kappa$-value $\frac{1}{3}$, and whose local rule $\phi$ is given in Table \ref{local-rule-table}.

\begin{table}[] \label{local-rule-table}
\begin{tabular}{|l||l|l|l|l|l|l|l|l|}
\hline   

  & a & b & c & d & e & f & g & h \\ \hhline{|=#=|=|=|=|=|=|=|=|}
a &   &   &   &   &   & b &   & b \\ \hline
b & f &   & f & f & h & h &   &   \\ \hline
c &   & a &   &   &   &   &   &   \\ \hline
d &   & g &   &   &   & e &   &   \\ \hline
e &   & f &   & f &   &   &   &   \\ \hline
f & d & d &   &  &b   &   & b & b \\ \hline
g &   &   &   &   &   & e &   &   \\ \hline
h &   & a &   &   &   & c &   &   \\ \hline
\end{tabular}
\vspace{0.5cm} 
\caption{
Local rule table for $\phi,$ Example \ref{twisted}. The entry in the row labelled with $\alpha$ and in  the column labelled with $\beta$ is $\phi(\alpha \beta)$.  The map $\phi$ need only be defined on the language of $\theta$.}
\end{table}

Let

\begin{align*}
	\eta(a) &= abfbh\\
	\eta(b) &= abfdg\\
	\eta(c) &= abfbh\\
	\eta(d) &= abfdg\\
	\eta(e) &= efbfd\\
	\eta(f) &= efbhc\\
	\eta(g) &= efbfd\\
	\eta(h) &= efbhc,
\end{align*}
Now $\eta$ has an order two automorphism $\tau$ defined by the permutation $(ae)(bf)(cg)(dh)$, 
$(X_\theta, \Phi)$ is conjugate to $ (X_{\eta}, \sigma)$, 
and $(X_{\theta}, \sigma)$ is conjugate to a twist (by $\tau$) of the  $3$-compression of $(X_\eta,\sigma)$.

\end{example}

In general, what can one say about the $\kappa$-kernel of the substitution $\theta$ versus that of the substitution $\eta$ in the statement of Theorem \ref{roots}? The following examples tell us that in general, there is no elegant statement. In this next example, we were still able to cling to a sense that we could set things up so that the $\kappa$-kernels of the relevant maps are isomorphic.

\begin{example}\label{example-salvaged}
Consider the substitution $\theta$ on the alphabet $\mathcal{A}$:
\begin{align*}
  \theta(A) &= ABC\\
  \theta(B) &= AAD\\
  \theta(C) &= EFG\\
  \theta(D) &= HIB\\
    \theta(E) &= EJJ\\
  \theta(F) &= FCB\\
  \theta(G) &= KKC\\
  \theta(H) &= HLL\\
 \theta(I) &= IDG\\
  \theta(J) &= FEL\\
  \theta(K) &= KGD\\
  \theta(L) &= IHJ.
\end{align*}

\begin{table}[]\label{example-salvaged-rule}

\begin{tabular}{|C||C|C|C|C|C|C|C|C|C|C|C|C|}
\hline
  & A  & B & C & D & E & F & G & H & I & J & K & L \\ \hhline{|=#=|=|=|=|=|=|=|=|=|=|=|=|}
A & J & J &   & F &   &   &   &   &   &   &   &   \\ \hline
B &   &   & I &   & I &   &   &   &   &   & L &   \\ \hline
C & H & H &   &   &   &   &  &D   &   &   &   &   \\ \hline
D &   &   &   &   & C &   & E &   &   &   & E &   \\ \hline
E &   &   &   &   &   & B &   &   &   & B &   & A \\ \hline
F &   &   & C &   & C &   & E &   &   &   &   &   \\ \hline
G & J &   &   & F &   &  &   &F   &   &   &   &   \\ \hline
H &   &   &   &   &   &   &   &   & G & K &   & G \\ \hline
I &   & H &   & D &   &   &   & D &   &   &   &   \\ \hline
J &   &   &   &   &   & K &   &   & G & K &   &   \\ \hline
K &   &   & I &   &   &   & L &   &   &   & L &   \\ \hline
L &   &   &   &   &   & B &   &   & A &   &   & A \\ \hline
\end{tabular}
\vspace{0.5cm} 
\caption{
The local rule for the square root $\Phi$ of the shift in Example \ref{example-salvaged}.}
\end{table}

Then $\theta$ has one nontrivial automorphism $\tau$ belonging to $\ker\kappa_\theta$, defined by the permutation $(AK)(BG)(CD)(EH)(FI)(JL)$. It also has two square roots of the shift, one of which has local rule in 
Table \ref{example-salvaged-rule} and which we call $\Phi$, the other of which is $\tau\circ\Phi$. Now using $\Phi$, and working through the proof of  Lemma \ref{roots-two}, we can show that  $(X_\theta,\Phi)$ is topologically conjugate to $(X_\eta,\sigma)$ where  
\begin{align*}
  \theta(a) &= acb\\
  \theta(a) &= bda\\
  \theta(c) &= cef\\
  \theta(d) &= dfe\\
    \theta(e) &= ead\\
  \theta(f) &= fbc.\\
\end{align*}

So $\theta$ is the twist of a compression of $\eta$, but  $\ker\kappa_\eta$ is isomorphic to $S_3$, so here it is not the case that the $\kappa$-kernels of the two maps $\theta$ and $\eta$ are the same.  However, note that if we use $\tau\circ\Phi$ in Lemma \ref{roots-two}, we obtain that 
$(X_\theta, \tau\circ\Phi)$ is topologically conjugate to $(X_{\eta_\tau},\sigma)$,  and that  $\ker \kappa_{\eta_\tau}$ {\em is} isomorphic to $\ker \kappa_{\theta}$.

\end{example}

 However, whatever hope we clung to with Example \ref{example-salvaged} was dashed with this next example.
 
 \begin{example}\label{no-hope}
Consider 

\begin{align*}
 \theta(A) &=ABC\\
\theta(B) &= DEC\\
\theta(C) &= AFG\\
\theta(D) &= DGB\\
\theta(E) &= HAC\\
\theta(F) &= DHG\\
\theta(G) &= HIB\\
\theta(H) &= HCG\\
\theta(\,I\,)  &= ADB.\\
\end{align*}

which is a two-compression of $(X_{\eta_1},\sigma)$ where 

\begin{align*}
\eta_1(a) &=abb, \\ 
\eta_1(b) &= bac, \\
\eta_1(c) &= cca. \\
\end{align*}

Note  that $\theta$ has a nontrivial kernel automorphism $\tau$ given by the permutation $(ADH)(BGC)(EIF)$, and so, since it is a two-compression, it also has three automorphisms with $\kappa$ value $1/2$, $\Phi_1$, $\Phi_2=\tau\circ\Phi_1$ and $\Phi_3=\tau^2\circ\Phi_1$. Note also that $\tau$ and $\Phi_1$ do not commute. Finally, $\Phi_1^{2} = \Phi_2^2=\Phi_3^2=\sigma$, and using $\Phi_i$ in Lemma \ref{roots-two}, we get that $(X_{\theta},\Phi_i)$ is conjugate to a two-compression of $(X_{\eta_i},\sigma)$ where 
\begin{align*}
\eta_1(a) &=abb, 
 \hspace{1cm}\eta_2(a)=acb  \hspace{1cm}   \eta_3(a) = aab\\ 
 \eta_1(b) &= bac \hspace{1.2cm}\eta_2(b) =bbc   \hspace{1cm} \eta_3(b) = bcc\\
\eta_1(c) &= cca   \hspace{1.2cm}    \eta_2(c)=cba  \hspace{1cm} \eta_3(c) =cba \\
\end{align*}
all three of whose automorphism groups are trivial.

\end{example}

\begin{remark}\label{obstruction}
What causes this breakdown? 
 We have seen in Theorem~\ref{characterisation-of-kernel} that the automorphisms of $\eta$ with $\kappa$-value $0$ can be characterized in terms of the minimal sets and the group
 \begin{align*}
 	G(\eta) := <\sigma_{M,j}: M \in \mathcal{X}, 0\leq j \leq r-1>.
 \end{align*}
 Let us now consider the $k$-compression  $\eta^{(k)}$ of $\eta$ for some $k \mid r-1$.
 One can then show that
 \begin{align*}
 	G(\eta^{(k)}) = <\sigma_{M,j}: M \in \mathcal{X}, 0 \leq j \leq r-1, k \mid j>.
 \end{align*}
 Thus, we might be in the situation where $G(\eta^{(k)})$ is strictly smaller then $G(\eta)$, or one where the centralizer of $G(\eta^{(k)})$ is strictly larger then the centralizer of $G(\eta)$.
 
 This is precisely what happens in this last example, where $G(\eta_i)=S_3$, so that $\eta_i$ has a trivial automorphism group, but $G(\eta_i^{(2)}) = A_3 = \{\Id, (123), (132)\} = C(G(\eta_i^{(2)}))$, so that we have non-trivial automorphisms with $\kappa$ value $0$. 
\end{remark}

\proof[Acknowledgements] 
We thank Anthony Quas for  the key insights in Corollary \ref{roots-corollary}, and Tomasz Downarowicz and Mariusz Lemanczyk for pointing out Theorem \ref{automorphism-almost-automorphic-v2} to us.
{\footnotesize
\bibliographystyle{abbrv}
\bibliography{bibliography}

\def\ocirc#1{\ifmmode\setbox0=\hbox{$#1$}\dimen0=\ht0 \advance\dimen0
  by1pt\rlap{\hbox to\wd0{\hss\raise\dimen0
  \hbox{\hskip.2em$\scriptscriptstyle\circ$}\hss}}#1\else {\accent"17 #1}\fi}
\begin{thebibliography}{10}

\bibitem{Baake-Lenz}
M.~Baake and D.~Lenz.
\newblock Spectral notions of aperiodic order.
\newblock {\em Discrete Contin. Dyn. Syst. Ser. S}, 10(2):161--190, 2017.

\bibitem{B-R-Y}
M.~Baake, J.~A.~G. Roberts, and R.~Yassawi.
\newblock Reversing and extended symmetries of shift spaces.
\newblock {\em Discrete Contin. Dyn. Syst.}, 38(2):835--866, 2018.

\bibitem{Cobham}
A.~Cobham.
\newblock Uniform tag sequences.
\newblock {\em Math. Systems Theory}, 6:164--192, 1972.

\bibitem{CDK}
E.~M. Coven, F.~M. Dekking, and M.~S. Keane.
\newblock Topological conjugacy of constant length substitution dynamical
  systems.
\newblock {\em Indag. Math. (N.S.)}, 28(1):91--107, 2017.

\bibitem{coven-quas-yassawi}
E.~M. Coven, A.~Quas, and R.~Yassawi.
\newblock Computing automorphism groups of shifts using atypical equivalence
  classes.
\newblock {\em Discrete Anal.}, pages Paper No. 3, 28, 2016.

\bibitem{dekking}
F.~M. Dekking.
\newblock The spectrum of dynamical systems arising from substitutions of
  constant length.
\newblock {\em Z. Wahrscheinlichkeitstheorie und Verw. Gebiete},
  41(3):221--239, 1977/78.

\bibitem{Downarowicz_Lemanczyk}
T.~Downarowicz and L.~M.
\newblock private communication, 2019.

\bibitem{Durand-Leroy}
F.~Durand and J.~Leroy.
\newblock Decidability of the isomorphism and the factorization between minimal
  substitution subshifts.

\bibitem{ellis-gottschalk-1960}
R.~Ellis and W.~H. Gottschalk.
\newblock Homomorphisms of transformation groups.
\newblock {\em Trans. Amer. Math. Soc.}, 94:258--271, 1960.

\bibitem{BDM}
F.~D. F.~Blanchard and A.~Maass.
\newblock {\em Nonlinearity}, 17:817--833, 2005.

\bibitem{Gottschalk-Hedlund}
W.~H. Gottschalk and G.~A. Hedlund.
\newblock {\em Topological dynamics}.
\newblock American Mathematical Society Colloquium Publications, Vol. 36.
  American Mathematical Society, Providence, R. I., 1955.

\bibitem{Herning}
J.~L. Herning.
\newblock {\em Spectrum and {F}actors of {S}ubstitution {D}ynamical {S}ystems}.
\newblock ProQuest LLC, Ann Arbor, MI, 2013.
\newblock Thesis (Ph.D.)--The George Washington University.

\bibitem{host}
B.~Host.
\newblock Valeurs propres des syst\`emes dynamiques d\'efinis par des
  substitutions de longueur variable.
\newblock {\em Ergodic Theory Dynam. Systems}, 6(4):529--540, 1986.

\bibitem{HP}
B.~Host and F.~Parreau.
\newblock Homomorphismes entre syst\`emes dynamiques d\'{e}finis par
  substitutions.
\newblock {\em Ergodic Theory Dynam. Systems}, 9(3):469--477, 1989.

\bibitem{kamae}
T.~Kamae.
\newblock A topological invariant of substitution minimal sets.
\newblock {\em J. Math. Soc. Japan}, 24:285--306, 1972.

\bibitem{L-M}
M.~Lema\'{n}czyk and M.~K. Mentzen.
\newblock On metric properties of substitutions.
\newblock {\em Compositio Math.}, 65(3):241--263, 1988.

\bibitem{L-M-2018}
M.~{Lema{\'n}czyk} and C.~{M{\"u}llner}.
\newblock {Automatic sequences are orthogonal to aperiodic multiplicative
  functions}.
\newblock {\em arXiv e-prints}, page arXiv:1811.00594, Nov 2018.

\bibitem{martin}
J.~C. Martin.
\newblock Substitution minimal flows.
\newblock {\em Amer. J. Math.}, 93:503--526, 1971.

\bibitem{Mosse}
B.~Moss{\'e}.
\newblock Puissances de mots et reconnaissabilit\'e des points fixes d'une
  substitution.
\newblock {\em Theoret. Comput. Sci.}, 99(2):327--334, 1992.

\bibitem{M-2017}
C.~M\"{u}llner.
\newblock Automatic sequences fulfill the {S}arnak conjecture.
\newblock {\em Duke Math. J.}, 166(17):3219--3290, 2017.

\bibitem{Queffelec}
M.~Queff\'{e}lec.
\newblock {\em Substitution dynamical systems---spectral analysis}, volume 1294
  of {\em Lecture Notes in Mathematics}.
\newblock Springer-Verlag, Berlin, second edition, 2010.

\bibitem{Staynova}
P.~{Staynova}.
\newblock {The Ellis Semigroup of a Generalised Morse System}.
\newblock {\em arXiv e-prints}, page arXiv:1711.10484, Nov 2017.

\end{thebibliography}
}

\end{document}